\documentclass[11pt, reqno]{amsart}
\usepackage{amsfonts,amsmath, amssymb,amsthm,amscd}
\usepackage[usenames,dvipsnames]{color}
\usepackage{bm}
\usepackage{mathrsfs}
\usepackage{yfonts}
\usepackage{verbatim}
\usepackage{hyperref}
\usepackage{graphicx}
\usepackage[utf8]{inputenc}
\usepackage{latexsym}
\usepackage{lscape}
\usepackage{epsfig}
\usepackage{subfigure}
\usepackage{dsfont}
\usepackage[left=2.5cm, right=2.5cm, top=3cm, bottom=3cm]{geometry}
\usepackage{tikz}

\usetikzlibrary{patterns}
\usepackage{setspace}

\usepackage{mathdots}
\usepackage{mathtools}
\usepackage{tikz}
\usetikzlibrary{shapes.multipart}
\usetikzlibrary{patterns}
\usetikzlibrary{shapes.multipart}
\usetikzlibrary{arrows}
\usetikzlibrary{decorations.markings}

\begin{document}
	\newcommand{\EE}{\ensuremath{\mathbb{E}}}
	\newcommand{\qq}[1]{(q;q)_{#1}}
	\newcommand{\A}{\ensuremath{\mathcal{A}}}
	\newcommand{\GT}{\ensuremath{\mathbb{GT}}}
	\newcommand{\PP}{\ensuremath{\mathbb{P}}}
	\newcommand{\R}{\ensuremath{\mathbb{R}}}
	\newcommand{\Rplus}{\ensuremath{\mathbb{R}_{+}}}
	\newcommand{\C}{\ensuremath{\mathbb{C}}}
	\newcommand{\Z}{\ensuremath{\mathbb{Z}}}
	\newcommand{\Weyl}[1]{\ensuremath{\mathbb{W}}^{#1}}
	\newcommand{\Zgzero}{\ensuremath{\mathbb{Z}_{>0}}}
	\newcommand{\Zgeqzero}{\ensuremath{\mathbb{Z}_{\geq 0}}}
	\newcommand{\Zleqzero}{\ensuremath{\mathbb{Z}_{\leq 0}}}
	\newcommand{\Q}{\ensuremath{\mathbb{Q}}}
	\newcommand{\T}{\ensuremath{\mathbb{T}}}
	\newcommand{\Y}{\ensuremath{\mathbb{Y}}}
	
	\newcommand{\Real}{\ensuremath{\mathfrak{Re}}}
	\newcommand{\Imag}{\ensuremath{\mathfrak{Im}}}
	\newcommand{\re}{\ensuremath{\mathfrak{Re}}}
	
	\newcommand{\Sym}{\ensuremath{\mathrm{Sym}}}
	
	\newcommand{\bfone}{\ensuremath{\mathbf{1}}}
	
	\newcommand{\edge}{\textrm{edge}}
	\newcommand{\dist}{\textrm{dist}}
	
	\def \Ai {{\rm Ai}}
	\def \sgn {{\rm sgn}}
	\newcommand{\var}{{\rm var}}
	
	\newcommand{\Res}[1]{\underset{{#1}}{\mathrm{Res}}}
	\newcommand{\Resfrac}[1]{\mathrm{Res}_{{#1}}}
	\newtheorem{theorem}{Theorem}[section]
	\newtheorem{partialtheorem}{Partial Theorem}[section]
	\newtheorem{conj}[theorem]{Conjecture}
	\newtheorem{lemma}[theorem]{Lemma}
	\newtheorem{proposition}[theorem]{Proposition}
	\newtheorem{corollary}[theorem]{Corollary}
	\newtheorem{claim}[theorem]{Claim}
	\newtheorem{formal}[theorem]{Critical point derivation}
	\newtheorem{experiment}[theorem]{Experimental Result}
	\newtheorem{question}{Question}
	
	\def\note#1{\textup{\textsf{\color{blue}(#1)}}}
	
	\def\noteG#1{\textup{\textsf{ \color{red!60!yellow}(#1)}}}
	
	\theoremstyle{definition}
	\newtheorem{remark}[theorem]{Remark}
	
	\theoremstyle{definition}
	\newtheorem{example}[theorem]{Example}

	\theoremstyle{definition}
	\newtheorem{definition}[theorem]{Definition}
	
	\theoremstyle{definition}
	\newtheorem{definitions}[theorem]{Definitions}

	\newcommand{\PMPasc}{\ensuremath{\mathbf{PMP_{asc}}}}
	\newcommand{\PSPasc}{\ensuremath{\mathbf{PSP_{asc}}}}
	\newcommand{\PSP}{\ensuremath{\mathbf{PSP}}}
	\newcommand{\PSM}{\ensuremath{\mathbf{PSM}}}
	\newcommand{\Plancherel}{\ensuremath{\mathrm{Plancherel}}}
	\newcommand{\Pf}{\ensuremath{\mathrm{Pf}}}
	\newcommand{\Prob}{ \ensuremath{\mathrm{Prob}}}
	\newcommand{\Geom}{\ensuremath{\mathrm{Geom}}}
	\newcommand{\sign}{\ensuremath{\mathrm{sign}}}
	\newcommand{\Conf}{\ensuremath{\mathrm{Conf}}}
	\newcommand{\I}{\ensuremath{\mathbf{i}}}
	\newcommand{\skeww}{\ensuremath{\mathrm{Skew}_2}}
	
	\newcommand{\kernel}{\mathsf{K}}
	\newcommand{\Ikernel}{\mathsf{I}}
	\newcommand{\Rkernel}{\mathsf{R}}
	
	\renewcommand{\leq}{\leqslant}
	\renewcommand{\geq}{\geqslant}
	
	\title{Facilitated exclusion process}

	\begin{abstract}
		We study the Facilitated TASEP, an interacting particle system on the one dimensional  integer lattice. We prove that starting from step initial condition, the position of the rightmost  particle has Tracy Widom GSE statistics on a cube root time scale, while the statistics in the bulk of the rarefaction fan are GUE. This uses a mapping with last-passage percolation in a half-quadrant which is exactly solvable through Pfaffian Schur processes.
		
		Our results further probe the question of how first particles fluctuate for exclusion processes with downward jump discontinuities in their limiting density profiles. Through the Facilitated TASEP and a previously studied MADM exclusion process we deduce that cube-root time fluctuations seem to be a common feature of such systems. However, the statistics which arise are shown to be model dependent (here they are GSE, whereas for the MADM exclusion process they are GUE).
		
		We also discuss a two-dimensional crossover between GUE, GOE and GSE distribution by studying the multipoint distribution of the first particles when the rate of the first one varies. In terms of half-space last passage percolation, this corresponds to last passage times close to the boundary when the size of the boundary  weights is simultaneously scaled close to the critical point.
	\end{abstract}
	
	\author[J. Baik]{Jinho Baik}
	\address{J. Baik, University of Michigan, Department of Mathematics,
		530 Church Street,
		Ann Arbor, MI 48109, USA}
	\email{baik@umich.edu}
	
	\author[G. Barraquand]{Guillaume Barraquand}
	\address{G. Barraquand,
		Columbia University,
		Department of Mathematics,
		2990 Broadway,
		New York, NY 10027, USA}
	\email{barraquand@math.columbia.edu}
	
	\author[I. Corwin]{Ivan Corwin}
	\address{I. Corwin, Columbia University,
		Department of Mathematics,
		2990 Broadway,
		New York, NY 10027, USA
	}
	\email{ivan.corwin@gmail.com}
	
	\author[T. Suidan]{Toufic Suidan}
	\address{T. Suidan}
	\email{tsuidan@gmail.com}
	
	\maketitle
	
	\setcounter{tocdepth}{1}
	\hypersetup{linktocpage}
	\tableofcontents

\section{Introduction}

Exclusion processes on $\Z$ are expected, under mild hypotheses, to belong to the KPZ universality class \cite{corwin2012kardar, halpin2015kpz}. As a consequence, one expects that if particles start densely packed from the negative integers -- the step initial condition -- the positions of particles in the bulk of the rarefaction fan will fluctuate on a cube-root time scale with GUE Tracy-Widom  statistics in the large time limit. The motivation for this paper is to consider the fluctuations of the location of the rightmost particle and probe its universality over different exclusion processes. 

In the totally asymmetric simple exclusion process (TASEP) the first particle jumps by $1$ after an exponentially distributed waiting time of mean $1$, independently of everything else. Hence its location satisfies a classical Central Limit Theorem when time goes to infinity (i.e. square-root time fluctuation with limiting Gaussian statistics). This is true for any totally asymmetric exclusion process starting from step initial condition. However, in the asymmetric simple exclusion process (ASEP), the trajectory of the first particle  is affected by the behaviour of the next particles. This results in a different limit theorem. Tracy and Widom showed \cite[Theorem 2]{tracy2009asymptotics}
that the fluctuations still occur on the $t^{1/2}$ scale, but the limiting distribution is different and depends on the strength of the asymmetry (see also \cite{lee2017distributions} where the same distribution arises for the first particle's position in a certain zero-range process). In \cite{barraquand2015q}, another partially asymmetric process called the MADM exclusion process was studied. The first particle there fluctuates  on a $t^{1/3}$ scale with Tracy-Widom GUE limit distribution, as if it was in the bulk of the rarefaction fan. An explanation for why the situation is so contrasted with ASEP (and other model where the first particle has the same limit behaviour) is that the MADM, when started from step initial condition, develops a downward jump discontinuity of its density profile around the first particle (see Figure 3 in \cite{barraquand2015q}). 

In this paper, we test the universality of the fluctuations of the first particle in the presence of a jump discontinuity -- does the $t^{1/3}$ scale and GUE statistics survive over other models? We solve this question for  the Facilitated TASEP.  Our results  show that the GUE distribution does not seem to survive in general, though we do still see the $t^{1/3}$ scale.

\subsection{The Facilitated TASEP}
The Facilitated Totally Asymmetric Simple Exclusion Process  (abbreviated FTASEP in the following) was introduced in \cite{basu2009active} and further studied in \cite{gabel2010facilitated, gabel2011cooperativity}. This is an interacting particle system on $\Z$, satisfying the exclusion rule, which means that each site is occupied by at most $1$ particle. A particle sitting at site $x$ jumps to the right by $1$ after an exponentially distributed waiting time of mean $1$, provided that the target site  (i.e. $x+1$) is empty and that the site $x-1$ is occupied. Informally, the dynamics are very similar with TASEP, with the only modification being particles need to wait to have a left neighbour (facilitation) before moving (See Figure \ref{fig:facilitatedintro}). It was introduced as a simplistic model for motion in glasses: particles move faster in less crowded areas (modelled by the exclusion rule), but need a stimulus to move (modelled by the facilitation rule). We focus here on the step initial condition: at time $0$, the particles occupy all negative sites, and the non-negative sites are empty. 

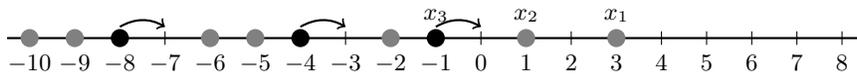
\begin{figure}
\begin{tikzpicture}[scale=0.6]
\draw[thick] (-9.5, 0) -- (9.5, 0);
\foreach \x in {-10, ..., 8} {
	\draw (\x+1, 0.15) -- (\x+1, -0.15) node[anchor=north]{\footnotesize ${\x}$};
}
\fill[gray] (-9, 0) circle(0.2);
\fill[gray] (-8, 0) circle(0.2);
\fill (-7, 0) circle(0.2);
\fill[gray] (-5, 0) circle(0.2);
\fill[gray] (-4, 0) circle(0.2);
\fill (-3, 0) circle(0.2);
\fill[gray] (-1, 0) circle(0.2);
\fill (0, 0) circle(0.2);
\draw (0, 0.5) node{\footnotesize $x_3$};
\fill[gray] (2, 0) circle(0.2);
\draw (2,0.5) node{\footnotesize $x_2$};
\fill[gray] (4, 0) circle(0.2);
\draw (4, 0.5) node{\footnotesize $x_1$};
\draw[thick, ->] (-7, 0.25) to[bend left] (-6, 0.25);
\draw[thick, ->] (-3, 0.25) to[bend left] (-2, 0.25);
\draw[thick, ->] (0, 0.25) to[bend left] (1, 0.25);
\end{tikzpicture}
\caption{The particles in black jump by $1$ at rate $1$ whereas particles in gray cannot. }
\label{fig:facilitatedintro}
\end{figure}

Since the dynamics preserve the order between particles, we can describe the configuration of the system by their ordered positions
$$ \dots <x_2 <x_1<\infty .$$
Let us collect some (physics) results from \cite{gabel2010facilitated} which studies  the hydrodynamic behaviour -- but not the fluctuations. Assume that the system is at equilibrium with an average  density of particles $\rho$. A family of translation invariant stationary measures indexed by the average density  -- conjecturally unique -- is described in the end of Section \ref{sec:defFTASEP}.  Then the flux, i.e. the average number of particles crossing a given bond per unit of time, is given by (see \cite[Eq. (3)]{gabel2010facilitated} and \eqref{eq:flux2} in the present paper)
\begin{equation}
j(\rho) = \frac{(1-\rho)(2\rho-1)}{\rho}.
\label{eq:flux}
\end{equation}
This is only valid when $\rho >1/2$. When $\rho<1/2$, \cite{gabel2010facilitated} argues  that the system eventually reaches a static state that consists of immobile single-particle clusters. One expects that the limiting density profile, informally given by
$$\rho(x, t):=\lim_{T\to\infty}\PP\big(\exists \text{ particle at site }xT\text{ at time }tT\big),$$
exists and is a weak solution subject to the entropy condition of the conservation equation
\begin{equation}
\frac{\partial}{\partial t}\rho(x,t) + \frac{\partial}{\partial x}j(\rho(x,t))=0.
\label{eq:PDEconservation}
\end{equation}
Solving this equation subject to the initial condition $ \rho(x, t)=\mathds{1}_{\lbrace x<0\rbrace}$ yields the density profile (depicted in Figure \ref{fig:densityprofile})
 $$\forall t >0, \ \rho(xt, t) = \begin{cases} 1 & \text{ if }x<-1, \\
 \frac{1}{\sqrt{2+x}}& \text{ if }-1\leqslant x\leqslant 1/4\\
 0 & \text{ if }x>1/4. \end{cases}$$
See also \cite[Eq. (5)]{gabel2010facilitated}.
\begin{figure}
\centering 
\begin{tikzpicture}[scale=3]
\draw[ ->] (-1.5, 0) -- (1.5, 0);
\draw[ ->] (0, -0.3) -- (0, 1.2);
\draw[thick] plot [smooth, domain=-1:0.25] (\x, {1/sqrt(2+ \x )});
\draw[thick] (-1.5, 1) -- (-1,1);
\draw[ultra thick] (0.25, 0) -- (1.45, 0);
\draw[dotted]   (0.25, 0) node[anchor = north] {$1/4$} -- (0.25, 2/3);
\draw (-1, 0.05) -- (-1, -0.05) node[anchor= north] {-1};
\end{tikzpicture}
\caption{Limiting density profile, i.e. graph of the function $x\mapsto  \rho(x, 1)$ .}
\label{fig:densityprofile}
\end{figure}
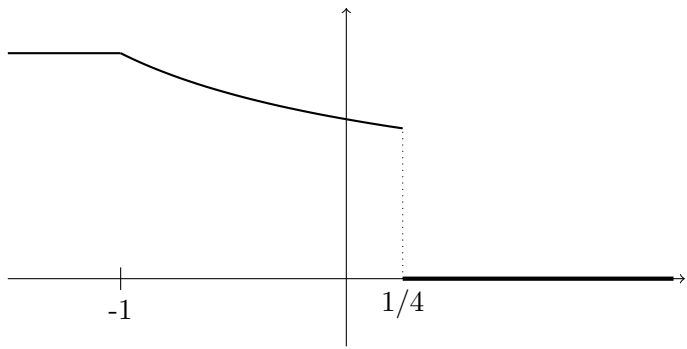
It is clear that there must be a jump discontinuity in the macroscopic density profile since in FTASEP particles can travel only in regions where the density is larger than $1/2$.

In general, the property of the flux which is responsible for the jump discontinuity is the fact that $j(\rho)/\rho$, i.e. the drift of a tagged particle, is not decreasing as a function of $\rho$. The density around the first particle will be precisely the value $\rho_{0}$ that maximizes the drift. Let us explain why. On one hand, the characteristics of PDEs such as  \eqref{eq:PDEconservation} are straight lines (\cite[§ 3.3.1.]{evans}), which means in our case that for any density $\bar{\rho}$ occurring in the rarefaction fan in the limit profile, there exists a constant $\pi(\bar{\rho})$ such that
\begin{equation}
\rho\big(\pi(\bar{\rho})t, t\big)=\bar{\rho}.
\label{eq:characteristics}
\end{equation}
$\pi(\bar{\rho})$ is the macroscopic position of particles around which the density is $\bar{\rho}$.
 Differentiating \eqref{eq:characteristics} with respect to $t$ and using the conservation equation \eqref{eq:PDEconservation} yields $\pi(\rho)= \frac{\partial j(\rho)}{\partial \rho}$. If we call $\rho_{0}$ the density around the first particle, then the macroscopic position of the first particle should be $\pi(\rho_{0})$. On the other hand, the first particle has a constant drift, which is\footnote{assuming local equilibrium -- which is not expected to be satisfied around the first particle but close to it-- the drift is given by $j(\rho)/\rho$ when the density is $\rho$.} $j(\rho_{0})/\rho_{0}$. Combining these observations yields
 $$ \frac{\partial j(\rho)}{\partial \rho}\bigg\vert_{\rho=\rho_0} =  \frac{j(\rho_{0})}{\rho_{0}}\ \ \ \text{i.e.} \ \ \ \frac{\mathrm{d}}{\mathrm{d}\rho}\frac{ j(\rho)}{ \rho}\bigg\vert_{\rho=\rho_0}=0.$$
This implies that a discontinuity of the density profile at the first particle can occur only if the drift is not strictly decreasing as a function of $\rho$, and
 it  suggests that $\rho_{0}$ is indeed the maximizer of the drift (see also \cite[Section 4]{barraquand2015q} for a different justification). In the example of the FTASEP, the maximum of
 $$ \frac{j(\rho)}{\rho} =  \frac{(1-\rho)(2\rho-1)}{\rho^2}$$
 is such that $ \rho_{0} = 2/3$ and $\pi(\rho_{0})=1/4$. In particular, this means that $x_1(t)/t$ should converge to $1/4$ when $t$ goes to infinity.

The fluctuations of $x_1(t)$ around $t/4$ are not GUE distributed as for the MADM exclusion process \cite[Theorem 1.3]{barraquand2015q}, but rather follow the GSE Tracy-Widom distribution in the large time limit.
\begin{theorem}
For FTASEP with step initial data,
$$ \PP\left(  \frac{x_{1}(t) - \frac t 4}{2^{-4/3} t^{1/3}} \geqslant x\right) \xrightarrow[t\to\infty]{} F_{\mathrm{GSE}}(- x),$$
where the GSE Tracy-Widom distribution function $F_{\mathrm{GSE}}$ is defined in Definition \ref{def:GSEdistribution}.
\label{th:FTASEPGSEintro}
\end{theorem}

In the bulk of the rarefaction fan, however, the locations of particles fluctuate as  the KPZ scaling theory predicts \cite{krug1992amplitude, spohn2012kpz}.
\begin{theorem}
For FTASEP with step initial data, and for any $r\in(0,1)$,
$$ \PP\left(  \frac{x_{\lfloor r t\rfloor}(t) - t\frac{1-6r+r^2}{4}}{ \varsigma t^{1/3}} \geqslant x\right) \xrightarrow[t\to\infty]{} F_{\mathrm{GUE}}(-x), $$
where $\varsigma = 2^{-4/3}\frac{(1+r)^{5/3}}{(1-r)^{1/3}}$ and the GUE Tracy-Widom distribution function $F_{\mathrm{GUE}}$ is defined in Section \ref{def:GUEdistribution}.
\label{th:FTASEPGUEintro}
\end{theorem}

We now consider a slightly more general process depending on a parameter $\alpha>0$ that we denote FTASEP($\alpha$), where the first particle jumps at rate $\alpha$ instead of $1$. We already know the nature of fluctuations of $x_1(t)$ when $\alpha=1$. It is natural to expect that fluctuations are still GSE Tracy-Widom distributed on the $t^{1/3}$ scale for $\alpha>1$. However, if $\alpha$ is very small, one expects that the first particle jumps according to a Poisson point process with intensity $\alpha$ and thus $x_1(t)$ has Gaussian fluctuations on the $t^{1/2}$ scale. It turns out that the threshold between these regimes happen when $\alpha=1/2$. 
\begin{theorem}
	Let $\vec{x}(t) = \lbrace x_n(t) \rbrace_{n\geqslant 1}$ be the particles positions in the FTASEP($\alpha$) started from step initial condition, when the first particle jumps at rate $\alpha$. Then,
	\begin{enumerate}
		\item For $\alpha >1/2$ ,
		$$ \PP\left(  \frac{x_{1}(t) - \frac t 4}{2^{-4/3} t^{1/3}} \geqslant x\right) \xrightarrow[t\to\infty]{} F_{\mathrm{GSE}}(- x).$$
		\item For $ \alpha =1/2$,
		$$ \PP\left(  \frac{x_{1}(t) - \frac t 4}{2^{-4/3} t^{1/3}} \geqslant x\right) \xrightarrow[t\to\infty]{} F_{\mathrm{GOE}}(- x).$$
		\item For $\alpha<1/2  $,
		$$ \PP\left(  \frac{x_{1}(t) - t\alpha(1-\alpha)}{\varsigma t^{1/2}} \geqslant x\right) \xrightarrow[t\to\infty]{} G(- x),$$
		where $G(x)$ is the standard Gaussian distribution function and
		$ \varsigma = \frac{1-2\alpha}{\sqrt{\alpha(1-\alpha)}}. $
	\end{enumerate}
	\label{thm:fluctuations}
\end{theorem}

It is also possible to characterize the joint distribution of several particles. An interesting case arises when we scale $\alpha$ close to the critical point and we look at particles indexed by $\eta t^{2/3}$ for different values of $\eta\geqslant 0$. More precisely, we scale 
	$$ \alpha = \frac{1+2^{4/3}\varpi \tau^{-1/3}}{2},$$
	where $\varpi \in \R$ is a free parameter and for any $\eta\geqslant 0$ consider the rescaled particle  position at time $t$ 
\begin{equation} X_{t}(\eta) := \frac{x_{2^{1/3}\eta t^{2/3}}(t) -\frac{t}{4}+ \eta \rho_0^{-1} 2^{1/3}t^{2/3} -\eta^2 2^{-4/3}}{t^{1/3} 2^{-4/3}},
\label{eq:defrescaledposition}
\end{equation}
 where $\rho_0=2/3$ (This is the density near the first particles in FTASEP($1$)). 
\begin{theorem}
	For any $p_1, \dots, p_k \in \R$,  and $0\leqslant \eta_1 < \dots < \eta_k$
	$$\lim_{t\to\infty} \PP\left( \bigcap_{i=1}^k  \left\lbrace X_{t}(\eta_i)  \geqslant p_i\right\rbrace  \right) = \Pf(J  - \kernel^{\rm cross})_{\mathbb{L}^2(\mathbb{D}_k(-p_1, \dots, -p_k))},$$
	where the right hand side is the Fredholm Pfaffian (see Definition \ref{def:FredholmPfaffian}) of some kernel  $\kernel^{\rm cross}$ (depending on $\varpi$ and the $\eta_i$) introduced in \cite[Section 2.5]{baik2017pfaffian} (see also Section \ref{sec:crossoverasymptotics} of the present paper) on the domain $\mathbb{D}_k(-p_1, \dots, -p_k)$ where 
	$$ \mathbb{D}_k(x_1, \dots, x_k) = \lbrace (i,x)\in \lbrace 1, \dots, k\rbrace\times\R: x \geqslant x_i\rbrace.$$
	\label{thm:crossFTASEP}
\end{theorem}

For the FTASEP, that is when $\alpha=1$ we have 
\begin{theorem}
	For any $p_1, \dots, p_k \in \R$,  and $0< \eta_1 < \dots < \eta_k$
	$$\lim_{t\to\infty} \PP\left( \bigcap_{i=1}^k  \left\lbrace X_{t}(\eta_i)  \geqslant p_i\right\rbrace  \right) = \Pf(J - \kernel^{\rm SU})_{\mathbb{L}^2(\mathbb{D}_k(-p_1, \dots, -p_k))},$$
	where the right hand side is the Fredholm Pfaffian of some kernel  $\kernel^{\rm SU}$ (depending on the $\eta_i$) introduced in \cite[Section 2.5]{baik2017pfaffian}  (see also Section \ref{sec:crossoverasymptotics}).
	\label{thm:SUFTASEP}
\end{theorem}

\subsection{Half-space last passage percolation}
\label{sec:halfspaceLPP}
Our route to prove Theorems \ref{th:FTASEPGSEintro}, \ref{th:FTASEPGUEintro}, \ref{thm:fluctuations}, \ref{thm:crossFTASEP} and \ref{thm:SUFTASEP} in Section \ref{sec:facilitated} uses a mapping with Last Passage Percolation (LPP) on a half-quadrant. 
\begin{definition}[Half-space exponential weight LPP] Let $\big(w_{n,m}\big)_{n\geqslant m\geqslant 0}$ be a sequence of  i.i.d. exponential random variables with rate $1$ (see Definition \ref{def:expdistribution}) when $n\geqslant m+1$ and with rate $\alpha$ when  $n=m$. We define the exponential last passage percolation time on the half-quadrant, denoted $H(n,m)$, by the recurrence for $n\geqslant m$,
 $$ H(n,m) =  w_{n,m} + \begin{cases}
  \max\Big\lbrace H(n-1, m)  ; H(n, m-1)\Big\rbrace &\mbox{if } n\geqslant m+1, \\
H(n,m-1) &\mbox{if } n=m
\end{cases}$$
with the boundary condition $H(n,0)=0$.
\label{def:LPPexp}
\end{definition}
We show in Proposition \ref{prop:coupling} that FTASEP is equivalent to a TASEP on the positive integers with a source of particles at the origin. We call the latter model half-line TASEP. The mapping between the two processes is the following: we match the gaps between consecutive particles in the FTASEP with the occupation variables in the half-line TASEP. Otherwise said, we study how the holes travel to the left in the FTASEP and prove that if one shrinks all distances between consecutive holes by one, the dynamics of holes follow those of the half-line TASEP (see the proof of Proposition \ref{prop:coupling}, in particular Figure \ref{fig:coupling}). In the case of full-space TASEP it is well-known that the height function of TASEP has the same law as the border of the percolation cluster of the LPP model with exponential weights (in a quadrant). This mapping remains true for half-line TASEP and LPP on the half-quadrant (Lemma \ref{lem:couplinglppTASEP}, see Figure \ref{fig:lastpassagehalfquadrant}).

 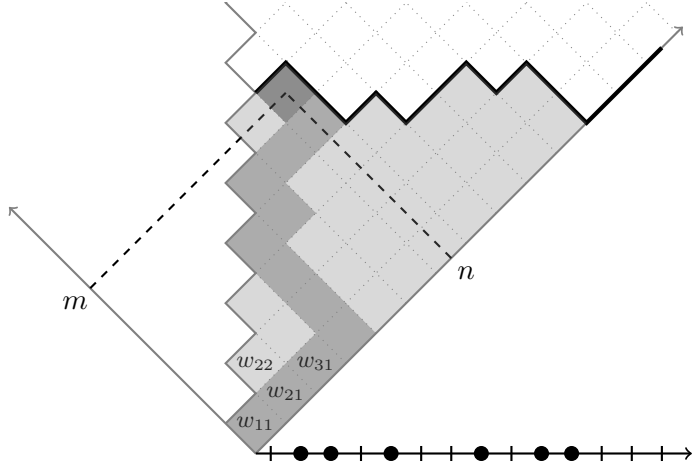
\begin{figure}
 	\begin{center}
 		\begin{tikzpicture}[scale=0.4]
 		\clip (-10,-1) rectangle(20,15);
 		\begin{scope}[rotate=45, scale=1.414]
 		\draw[thick, gray,  ->] (0,0) -- (14.2,0);
 		\draw[thick, gray, ->] (0,0) -- (0,8.2);
 		\draw[thick, gray] (0,1) -- (1,1) -- (1,2) -- (2,2) -- (2,3) -- (3,3) -- (3,4) -- (4,4) -- (4,5) -- (5,5) -- (5,6) -- (6,6) -- (6,7) -- (7,7) -- (7,8) -- (8,8)  ; 
 		\foreach \x in {1, ..., 8}
 		\draw[dotted, gray] (\x,0) -- (\x,\x);
 		\foreach \x in {9, 10, 11, 12, 13, 14}
 		\draw[dotted, gray] (\x,0) -- (\x,8);
 		\foreach \x in {1, ..., 8}
 		\draw[dotted, gray] (\x,\x) -- (14.1,\x);
 		\fill[gray, opacity=0.5] (0,0) -- (0,1) -- (3,1) -- (3,2) -- (3,3) -- (3,4) -- (4,4) -- (4,5) -- (5,5) -- (6,5) -- (6,6) -- (7,6) -- (7,5)-- (7,4) -- (6,4) -- (5,4) -- (5,3) -- (4,3) -- (4,2) -- (4,0) --  (0,0);
 		\fill[gray, opacity=0.7] (6,6) -- (7,6) -- (7,5) -- (6,5) -- (6,6);
 		\node at (6.5, -0.5) {$n$};
 		\node at (-0.5, 5.5) {$m$};
 		\draw [thick, dashed] (0,5.5) -- (6.5,5.5);
 		\draw [thick, dashed] (6.5,0) -- (6.5,5.5);
 		\node at (0.5, 0.5) {\footnotesize{$w_{11}$}};
 		\node at (1.5, 0.5) {\footnotesize{$w_{21}$}};
 		\node at (1.5, 1.5) {\footnotesize{$w_{22}$}};
 		\node at (2.5, 0.5) {\footnotesize{$w_{31}$}};
 		\draw[ultra thick, black] (6,6) -- (7,6) -- (7,4) -- (8,4) -- (8,3) -- (10, 3) -- (10,2) -- (11,2) -- (11,0) -- (13.5,0) ;
 		\fill[gray, opacity=0.3] (0,0)-- (0,1) -- (1,1) -- (1,2) -- (2,2) -- (2,3) -- (3,3) -- (3,4) -- (4,4) -- (4,5) -- (5,5) -- (5,6) -- (6,6) -- (7,6) -- (7,4) -- (8,4) -- (8,3) -- (10, 3) -- (10,2) -- (11,2) -- (11,0);
 		\end{scope}
 		
 		\draw[thick,  ->] (0,0) -- (14.5,0);
 		\foreach \x in {0, 3, 5, 6, 8, 11, 12, 13}
 		\draw[thick] (\x +0.5,-0.25) -- (\x +0.5,0.25);
 		\foreach \x in {1, 2, 4, 7, 9, 10}
 		{
 			\fill[thick] (\x + 0.5, 0) circle(0.25);}
 		\end{tikzpicture}
 	\end{center}
 	\caption{LPP on the half-quadrant. One admissible path from $(1,1)$ to $(n,m)$ is shown in dark gray. $H(n,m)$ is the maximum over such paths of the sum of the weights $w_{ij}$ along the path. The light gray area corresponds to the percolation cluster at some fixed time, and its border (shown in black) is associated with the particle system depicted on the horizontal line.}
 	\label{fig:lastpassagehalfquadrant}
 \end{figure}

The advantage of this mapping between FTASEP and half-space  last-passage percolation is that we can now use limit theorems proved for the latter (see \cite{baik2017pfaffian} and references therein), which we recall below. 
\begin{theorem}[{\protect\cite[Theorem 1.4]{baik2017pfaffian}}] The last passage time on the diagonal $H(n,n)$ satisfies the following limit theorems, depending on the rate $\alpha$ of the weights on the diagonal.
\begin{enumerate}
\item For $\alpha >1/2$,
$$ \lim_{n\to\infty} \PP\left( \frac{H(n,n) -4n}{2^{4/3}n^{1/3}} < x \right) = F_{\rm GSE}\left( x\right).$$
\item For $\alpha =1/2$,
$$ \lim_{n\to\infty} \PP\left( \frac{H(n,n) -4n}{2^{4/3}n^{1/3}} < x \right) = F_{\rm GOE}\left( x\right),$$
where the GOE Tracy-Widom distribution function $F_{\mathrm{GOE}}$ is defined in Lemma  \ref{def:GOEdistribution}.
\item For $\alpha <1/2$,
$$ \lim_{n\to\infty} \PP\left( \frac{H(n, n) -\frac{n}{\alpha(1-\alpha)}}{\sigma n^{1/2}} < x \right)  = G(x),$$
where $G(x)$ is the probability distribution function of the standard Gaussian, and
$$\sigma^2= \frac{1-2\alpha}{\alpha^2 (1-\alpha)^2}.$$
\end{enumerate}
\label{theo:LPPdiagointro}
\end{theorem}

Away from the diagonal, the limit theorem satisfied by $H(n,m)$ happens to be exactly the same as in the unsymmetrized or full-space model.
\begin{theorem}[{\protect\cite[Theorem 1.5]{baik2017pfaffian}}]
For any $\kappa \in (0,1)$ and $\alpha >\frac{\sqrt{\kappa}}{1+\sqrt{\kappa}}$, we have that when $m=\kappa n+sn^{2/3-\epsilon}$, for any $s\in \R$ and $\epsilon>0$, 
$$ \lim_{n\to\infty} \PP\left( \frac{H(n,m) -(1+\sqrt{\kappa})^2 n}{\sigma n^{1/3}} < x \right) = F_{\rm GUE}(x),$$
where
$$ \sigma = \frac{(1+ \sqrt{\kappa})^{4/3}}{\sqrt{\kappa}^{1/3}}.$$
\label{theo:LPPawaydiagointro}
\end{theorem}
In \cite{baik2017pfaffian}, we also explained how to obtain a two dimensional crossover between all the above cases by tuning the parameters $\alpha$ and $\kappa$ close to their critical value in the scale $n^{-1/3}$ (see Figure \ref{fig:phasediagram}). The proofs of the following results were omitted in \cite{baik2017pfaffian} (they were stated as Theorem 1.8 and 1.9 in \cite{baik2017pfaffian}) and we include them in Section \ref{sec:crossoverasymptotics}. 
\begin{figure}
\begin{tikzpicture}[scale=0.63]
\fill[fill=gray!15] (0,0) -- plot [smooth, domain=0:4.5]  (\x, {4/(sqrt(4*\x+1)+1)}) --  (4.5, 0) -- cycle;

\draw[ gray, >=stealth'] (0,-0.2) -- (0,5);
\draw[ gray, >=stealth'] (-0.2,0) -- (5, 0);
\fill[] (0,2) circle(0.07);
\draw[thick] plot [smooth, domain=0:4.5]  (\x, {4/(sqrt(4*\x+1)+1)});

\draw (-1, 0) node {{\footnotesize $\alpha=0$}};
\draw (-1, 2) node {{\footnotesize $\alpha=1/2$}};
\draw (-1, 5) node {{\footnotesize $\alpha=+\infty$}};

\draw (0, -0.5) node {{\footnotesize $\frac n m=1$}};
\draw (5, -0.5) node {{\footnotesize $\frac n m=+\infty$}};

\draw (0, 3.5) node[fill=white, opacity=.6,text opacity=1]  {$\mathrm{GSE}$};
\draw (2.5, 2.5) node[fill=white, opacity=.6,text opacity=1]  {$\mathrm{GUE}$};
\draw (0.5, 0.5) node[fill=white, opacity=.6,text opacity=1]  {$\mathrm{Gaussian}$};
\draw (2.5, 1) node[fill=white, opacity=.6,text opacity=1] {$\mathrm{GOE}^2$};

\draw[ dotted] (-0.2, 1.5) -- (0.5, 1.5) -- (0.5, 2.5) -- (-0.2, 2.5) -- cycle;
\draw[ dotted] (0.5, 1.5) -- (8,-3);
\draw[ dotted] (0.5, 2.5) -- (8,7);
\draw[dotted] (8, -3) -- (16, -3) -- (16, 7) -- (8, 7) -- cycle;

\draw[gray] (5.5, 4.2) node {{\footnotesize $ \alpha=\frac{1+2^{2/3}\varpi k^{-1/3}}{2} $}};
\draw[gray] (5.5, 3.4) node {{\footnotesize $n = k + 2^{2/3}k^{2/3}\eta $}};
\draw[gray] (5.5, 2.8) node {{\footnotesize $m = k - 2^{2/3}k^{2/3}\eta $}};

\draw[->, gray, >=stealth'] (9,-2) -- (9,6);
\draw[->, gray, >=stealth'] (9,2) -- (14, 2);
\draw[thick, ->, >=stealth'] plot [domain=9:13] (\x, 11-\x);
\draw (14, 1.5) node {$ \eta $};
\draw (8.5, 5.5) node {$ \varpi $};

\draw (9, 6.5) node {$\mathrm{GSE}$};
\draw (9, 2) node[fill=white, opacity=.6,text opacity=1]  {$\mathrm{GOE}$};
\draw (9.5, -2.5) node {$\mathrm{Gaussian}$};
\draw (15, 2) node {$\mathrm{GUE}$};
\draw (13.7, -2.5) node{$\mathrm{GOE}^2$};
\draw (11.5, 3) node {$\mathcal{L}_{\rm cross}(\cdot; \varpi, \eta)$};

\draw[gray, ->, >=stealth'] (13, 5.2) -- (13.5, 5.7);
\draw (14, 6) node {$\mathrm{GUE}$};
\draw[gray, ->, >=stealth'] (12.9, 0) -- (13.5,-0.4) ;
\draw (14.2, -0.7) node {$\mathrm{GUE}$};
\end{tikzpicture}
\caption{Phase diagram of the fluctuations of $H(n,m)$ as $n\to\infty$ when $\alpha$ and the ratio $n/m$ varies. The gray area corresponds to a region of the parameter space where the fluctuations are on the scale $n^{1/2}$ and Gaussian. The bounding $\mathrm{GOE}^2$ curve asymptotes to zero as $n/m$ goes to $+\infty$. The crossover distribution $\mathcal{L}_{\rm cross}(\cdot; \varpi, \eta)$ is defined in \cite[Definition 2.9]{baik2017pfaffian} and describes the fluctuations in the vicinity of $n/m=1$ and $\alpha=1/2$.}
\label{fig:phasediagram}
\end{figure}
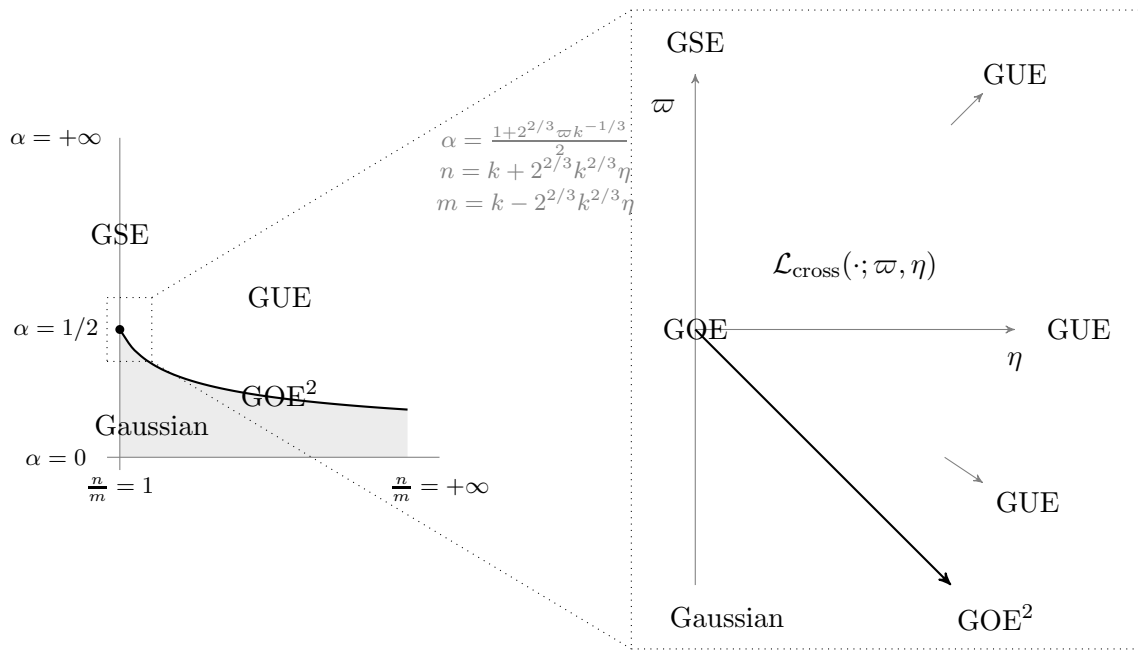
Let us define 
 $$H_n(\eta)= \frac{H\big(n + n^{2/3}\xi\eta , n -n^{2/3}\xi\eta\big)-4n + n^{1/3}\xi^2\eta^2}{\sigma n^{1/3}},$$
 where $\eta\geqslant 0$,  $\sigma = 2^{4/3}$  and $\xi=2^{2/3}$. We scale $\alpha$ as
 $$ \alpha = \frac{1+2\sigma^{-1}\varpi n^{-1/3}}{2} $$
 where  $\varpi\in\R$ is a free parameter. 
 \begin{theorem} For $0\leqslant \eta_1 < \dots < \eta_k$,  $\varpi\in\R$, 
 	$$ \lim_{n\to\infty} \PP\left(\bigcap_{i=1}^k \left\lbrace H_n(\eta_i) < h_i  \right\rbrace\right)  = \Pf\big( J- \kernel^{\rm  cross}\big)_{\mathbb{L}^2(\mathbb{D}_k(h_1, \dots, h_k))},$$
 	where $\kernel^{\rm cross}$ is defined in  Section \ref{sec:crossoverasymptotics}.
 	\label{theo:crossfluctuations}
 \end{theorem}
We refer to \cite[Sections 1.5 and 2.5]{baik2017pfaffian} for comments and explanations  about this kernel and its various degenerations. The phase diagram of one-point fluctuations is represented on Figure \ref{fig:phasediagram}.

 In the case when $\alpha>1/2$ is fixed, the joint distribution of passage-times is governed by the so-called symplectic-unitary transition \cite{forrester1999correlations}.
 \begin{theorem}
 	For $\alpha>1/2$ and $0< \eta_1< \dots <\eta_k$, 
 	we have that
 	$$ \lim_{n\to\infty} \PP\left(\bigcap_{i=1}^k \left\lbrace H_n(\eta_i) < h_i  \right\rbrace\right) = \Pf\big( J- \kernel^{\rm  SU}\big)_{\mathbb{L}^2(\mathbb{D}_k(h_1, \dots, h_k))},$$
 	where $\kernel^{\rm SU}$ is a certain matrix kernel introduced in \cite{baik2017pfaffian} (See also Section \ref{sec:crossoverasymptotics}).
 	\label{theo:SU}
 \end{theorem}
 Theorem \ref{theo:SU} corresponds to the  $ \varpi\to+\infty$ degeneration of Theorem \ref{theo:crossfluctuations}.

\subsection*{Outline of the paper}

In Section \ref{sec:defdistributions}, we provide the precise definitions of all probability distributions arising in this paper.
In Section \ref{sec:facilitated}, we explain the mapping between FTASEP and TASEP on a half-space with a source, or equivalently exponential LPP on a half-space. We prove  the limit theorems for the fluctuations of particles positions in FTASEP$(\alpha)$  using the asymptotic results for half-space LPP.
In Section \ref{sec:kpointdistribution}, we  recall the $k$-point distribution along space-like paths in half-space LPP with exponential weights (Proposition \ref{prop:kernelexponential}), derived in \cite{baik2017pfaffian}. In Section \ref{sec:crossoverasymptotics}, we provide a rigorous derivation of Theorem \ref{theo:crossfluctuations} and  \ref{theo:SU} from Proposition \ref{prop:kernelexponential}. This boils down to an asymptotic analysis of the correlation kernel that was omitted in \cite{baik2017pfaffian}. 

\subsection*{Ackowledgements}
We graciously acknowledge Zongrui Yang's help in catching and fixing some mistakes in our formulas for the crossover kernel in Section \ref{sec:crossoverasymptotics} of this paper (see the preface of Section \ref{sec:crossoverasymptotics} for further information). G.B and I.C. would like to thank Sidney Redner for drawing their attention to the Facilitated TASEP. Part of this work was done during the stay of J.B, G.B and I.C. at the Kavli Institute of Theoretical Physics and supported by the National Science Foundation under Grant No. NSF PHY11-25915. J.B. was supported in part by NSF grant DMS-1361782, DMS- 1664692 and DMS-1664531, and the Simons Fellows program. G.B. was partially supported by the Laboratoire de Probabilit\'es et Mod\`eles Al\'eatoires UMR CNRS 7599, Universit\'e Paris-Diderot--Paris 7  and the Packard Foundation through I.C.'s Packard Fellowship. I.C. was partially supported by the NSF through DMS-1208998 and DMS-1664650, the Clay Mathematics Institute through a Clay Research Fellowship, the Institute Henri Poincaré through the Poincaré Chair, and the Packard Foundation through a Packard Fellowship for Science and Engineering.

\section{Definitions of distribution functions}
\label{sec:defdistributions}
In this section, we provide definitions of the probability distributions arising in the paper.
\begin{definition}
The exponential distribution with rate $\alpha\in(0,+\infty)$, denoted $\mathcal{E}(\alpha)$,  is the probability distribution on $\R_{>0}$ such that if $X\sim \mathcal{E}(\alpha)$,
$$ \forall x\in \R_{>0}, \ \PP(X>x) = e^{-\alpha x}. $$
 \label{def:expdistribution}
\end{definition}
Let us introduce a convenient notation that we use throughout the paper to specify integration contours in the complex plane.
\begin{definition}
Let $\mathcal{C}_{a}^{\varphi}$ be the union of two semi-infinite rays departing $a\in\C$ with angles $\varphi$ and $-\varphi$. We  assume that the contour is oriented from $a+\infty e^{-i\varphi}$ to $a+\infty e^{+i\varphi}$.
\label{def:basicrays}
\end{definition}

We recall that for an integral operator $\mathcal{K}$ defined by a kernel $\mathsf{K}:\mathbb{X}\times \mathbb{X} \to \R,$
its Fredholm determinant $ \det(I+\mathcal{K})_{\mathbb{L}^2(\mathbb{X},\mu)}$ is given by the series expansion
$$ \det(I+\mathcal{K})_{\mathbb{L}^2(\mathbb{X},\mu)} = 1+\sum_{k=1}^{\infty}\frac{1}{k!} \int_{\mathbb{X}} \dots \int_{\mathbb{X}} \det\big( \mathsf{K}(x_i, x_j)\big)_{i,j=1}^k\ \mathrm{d}\mu^{\otimes k}(x_1 \dots x_k), $$
whenever it converges. Note that we will omit the measure $\mu$ in the notations and write simply  $\mathbb{L}^2(\mathbb{X})$ when the uniform or the Lebesgue  measure is considered. With a slight abuse of notations, we will also write $ \det(I+\mathsf{K})_{\mathbb{L}^2(\mathbb{X})}$ instead of $ \det(I+\mathcal{K})_{\mathbb{L}^2(\mathbb{X})}$.

\begin{definition}
The GUE Tracy-Widom distribution, denoted $\mathcal{L}_{\rm GUE}$ is a probability distribution on $\R$ such that if $X\sim \mathcal{L}_{\rm GUE}$,
$$ \PP(X\leqslant x) = F_{\rm GUE}(x) = \det(I-\kernel_{\rm Ai})_{\mathbb{L}^2(x,+\infty )} $$
where  $\kernel_{\rm Ai}$ is the Airy kernel,
\begin{equation}
\kernel_{\rm Ai} (u, v) = \int_{\mathcal{C}_{-1}^{2\pi/3}} \frac{\mathrm{d}w}{2\I\pi} \int_{\mathcal{C}_1^{\pi/3}} \frac{\mathrm{d}z}{2\I\pi} \frac{e^{z^3/3-zu}}{e^{w^3/3-wv}}\frac{1}{z-w}.
\label{eq:defairykernel}
\end{equation}
\label{def:GUEdistribution}
\end{definition}

In order to define the GOE and GSE distribution in a form which is convenient for later purposes, we introduce the concept of Fredholm Pfaffian.  
\begin{definition}[{\protect\cite[Section8]{rains2000correlation}}]
For a $2\times 2$-matrix valued  skew-symmetric kernel,
$$ \kernel(x,y) = \begin{pmatrix}
\kernel_{11}(x,y) & \kernel_{12}(x,y)\\
\kernel_{21}(x, y) & \kernel_{22}(x,y)
\end{pmatrix},\ \ x,y\in \mathbb{X},$$
we define its Fredholm Pfaffian by the series expansion
\begin{equation}
\Pf\big(J+\kernel \big)_{\mathbb{L}^2(\mathbb{X},\mu)} = 1+\sum_{k=1}^{\infty} \frac{1}{k!}
\int_{\mathbb{X}} \dots \int_{\mathbb{X}}  \Pf\Big( K(x_i, x_j)\Big)_{i,j=1}^k \mathrm{d}\mu^{\otimes k}(x_1 \dots x_k),
\label{eq:defFredholmPfaffian}
\end{equation}
provided the series converges, and we recall that for an skew-symmetric  $2k\times 2k$ matrix $A$, its Pfaffian is defined by
\begin{equation}
 \Pf(A) = \frac{1}{2^k k!} \sum_{\sigma\in\mathcal{S}_{2k}} \sign(\sigma) a_{\sigma(1)\sigma(2)}a_{\sigma(3)\sigma(4)} \dots a_{\sigma(2k-1)\sigma(2k)}.
 \label{def:pfaffian}
\end{equation}
The kernel $J$ is defined by 
$$ J(x,y)  = \begin{cases}
\begin{pmatrix}
0 & 1\\-1 & 0
\end{pmatrix} & \text{ if }x=y, \\
0& \text{ if }x\neq y.
\end{cases}$$
\label{def:FredholmPfaffian}
\end{definition}
In Sections \ref{sec:crossoverasymptotics}, we will need to control the convergence of Fredholm Pfaffian series expansions. This can be done using Hadamard's bound.
\begin{lemma}[{\protect\cite[Lemma 2.5]{baik2017pfaffian}}]
Let $\kernel(x, y)$ be a $2\times 2$ matrix valued skew symmetric kernel. Assume that there exist constants $C>0$ and constants $a>b\geqslant 0$ such that
\begin{equation*}
\vert \kernel_{11}(x, y) \vert <C e^{-ax-ay}, \ \
\vert \kernel_{12}(x, y)\vert  = \vert \kernel_{21}(y,x)\vert  <C e^{-ax+by}, \ \
\vert \kernel_{22}(x, y)\vert  <C e^{bx+by}.
\end{equation*}
Then, for all $k\in \Z_{>0}$,
$$ \Big\vert \Pf\big[ \kernel(x_i, x_j)\big]_{i,j=1}^k \Big\vert < (2k)^{k/2}C^k\prod_{i=1}^{k}e^{-(a-b) x_i}.$$
\label{lem:hadamard}
\end{lemma}

\begin{definition}
The GOE Tracy-Widom distribution, denoted $\mathcal{L}_{\rm GOE}$, is a continuous probability distribution on $\R$ whose  cumulative distribution function $F_{\rm GOE}(x)$ (i.e. $\PP(X\leqslant x)$ where $X\sim \mathcal{L}_{\rm GOE}$) is given by 
$$ F_{\rm GOE}(x)= \Pf\big( J- \kernel^{\rm GOE}\big)_{\mathbb{L}^2(x, \infty)},$$
where $\kernel^{\rm GOE}$ is the $2\times 2$ matrix valued kernel defined by
\begin{align*}
\kernel_{11}^{\rm GOE}(x,y) &=  \int_{\mathcal{C}_{1}^{\pi/3}}\frac{\mathrm{d}z}{2\I\pi}\int_{\mathcal{C}_{1}^{\pi/3}}\frac{\mathrm{d}w}{2\I\pi} \frac{z-w}{z+w} e^{z^3/3 + w^3/3 - xz -yw},\\
\kernel_{12}^{\rm GOE}(x,y) &= -\kernel_{21}^{\rm GOE}(x,y) = \int_{\mathcal{C}_{1}^{\pi/3}}\frac{\mathrm{d}z}{2\I\pi}\int_{\mathcal{C}_{-1/2}^{\pi/3}}\frac{\mathrm{d}w}{2\I\pi}  \frac{w-z}{2w(z+w)} e^{z^3/3 + w^3/3 - xz -yw} ,\\
\kernel_{22}^{\rm GOE}(x,y) &=  \int_{\mathcal{C}_{1}^{\pi/3}}\frac{\mathrm{d}z}{2\I\pi}\int_{\mathcal{C}_{1}^{\pi/3}}\frac{\mathrm{d}w}{2\I\pi}  \frac{z-w}{4zw(z+w)} e^{z^3/3 + w^3/3 - xz -yw} \\
& + \int_{\mathcal{C}_{1}^{\pi/3}}\frac{\mathrm{d}z}{2\I\pi} \frac{e^{z^3/3-zx}}{4z} -  \int_{\mathcal{C}_{1}^{\pi/3}} \frac{\mathrm{d}z}{2\I\pi} \frac{e^{z^3/3-zy}}{4z} -\frac{\sgn{(x-y)}}{4},
\end{align*}
where $\sgn{(x)} = \mathds{1}_{x>0} - \mathds{1}_{x<0}$.
\label{def:GOEdistribution}
\end{definition}

\begin{definition}
The GSE Tracy-Widom distribution, denoted $\mathcal{L}_{\rm GSE}$, is a continuous probability distribution on $\R$ whose  cumulative distribution function $F_{\rm GOE}$ is given by 
$$ F_{\rm GSE}\left( x\right) = \Pf\big( J- \kernel^{\rm GSE}\big)_{\mathbb{L}^2(x, \infty)} ,$$
where $\kernel^{\rm GSE}$ is a $2\times 2$-matrix valued kernel
defined by
\begin{align*}
\kernel_{11}^{\rm GSE}(x,y) &=  \int_{\mathcal{C}_{1}^{\pi/3}}\frac{\mathrm{d}z}{2\I\pi}\int_{\mathcal{C}_{1}^{\pi/3}}\frac{\mathrm{d}w}{2\I\pi} \frac{z-w}{4zw(z+w)} e^{z^3/3 + w^3/3 - xz -yw} ,\\
\kernel_{12}^{\rm GSE}(x,y) &= -\kernel_{21}^{\rm GSE}(x,y) =  \int_{\mathcal{C}_{1}^{\pi/3}}\frac{\mathrm{d}z}{2\I\pi}\int_{\mathcal{C}_{1}^{\pi/3}}\frac{\mathrm{d}w}{2\I\pi} \frac{z-w}{4z(z+w)} e^{z^3/3 + w^3/3 - xz -yw} ,\\
\kernel_{22}^{\rm GSE}(x,y) &= 
\int_{\mathcal{C}_{1}^{\pi/3}}\frac{\mathrm{d}z}{2\I\pi}\int_{\mathcal{C}_{1}^{\pi/3}}\frac{\mathrm{d}w}{2\I\pi} \frac{z-w}{4(z+w)} e^{z^3/3 + w^3/3 - xz -yw} \mathrm{d}z \mathrm{d}w.
\end{align*}
\label{def:GSEdistribution}
\end{definition}

\section{Facilitated totally asymmetric simple exclusion process}
\label{sec:facilitated}
\subsection{Definition and coupling}
\label{sec:defFTASEP}
A  configuration of particles on a subset $X$ of $\Z$ can be described either by occupation variables, i.e. a collection  $\vec{\eta} = ( \eta_x)_{x\in X}$ where $\eta_x=1$ if the site $x$ is occupied and $\eta_x=0$ else, or a vector of particle positions $\vec{x} = (x_i)_{i\in I}$ where the particles are indexed by some set $I$.  We will use both notations. 
\begin{definition}
	The FTASEP is a continuous-time Markov process defined on the state space $\lbrace 0,1\rbrace^{\Z}  $  via its Markov generator, acting on local functions $f: \lbrace 0,1\rbrace^{\Z}  \to \R$ by
	$$ L f(\vec{\eta}) = \sum_{x\in \Z}  \eta_{x-1} \eta_x (1-\eta_{x+1}) \big( f(\vec{\eta}_{x, x+1}) - f(\vec{\eta})\big),$$
	where the state $\vec{\eta}_{x, x+1}$ is obtained from $\vec{\eta}$ by exchanging occupation variables at sites $x$ and $x+1$.
\end{definition}
That this generator defines a Markov process corresponding to the particle dynamics described in the introduction can be justified, for instance, by checking the conditions of  \cite[Theorem 3.9]{liggett2005interacting}. 

We will be mostly interested in initial configurations that are right-finite, which means that there exists a right-most particle. Since the dynamics preserves the order between particles, it is convenient to alternatively describe a configuration of particles by their positions
$$ \dots <x_2 <x_1<\infty .$$
We also consider a more general version of the process where the first particle jumps at rate $\alpha$, while all other particles jump at rate $1$, and denote this process FTASEP($\alpha$).  Let us define state spaces corresponding to configurations of particles in FTASEP($\alpha$) where the distance between consecutive particles is at most $2$:
$$ \mathbb{X}_{>0} := \Big\lbrace (x_i)_{i\in \Z_{>0}} \in \Z^{ \Z_{>0}}:\  \forall i\in \Z_{>0},\  x_i-x_{i+1}-1 \in \lbrace 0,1\rbrace \Big\rbrace,$$
and 
$$ \mathbb{X}_{} := \Big\lbrace (x_i)_{i\in \Z}\in \Z^{ \Z} :\  \forall i\in \Z,\  x_i-x_{i+1}-1 \in \lbrace 0,1\rbrace \Big\rbrace.$$
Because of the facilitation rule, it is clear that the FTASEP($\alpha$) dynamics preserve both state spaces. 
\begin{definition}
For $\alpha>0$, the FTASEP($\alpha$) is a continuous-time Markov process defined on the state space $\mathbb{X}_{>0}$  via its Markov generator, acting on local functions $f: \mathbb{X}_{>0}  \to \R$ by
$$\mathcal{ L}^{FTASEP}_{\alpha} f(\vec{x}) =  \alpha\mathds{1}_{x_1 - x_{2}=1}\big( f(\vec{x}_{1}^+) - f(\vec{x})\big) +   \sum_{i\geqslant 2}  \mathds{1}_{x_i - x_{i+1}=1}\mathds{1}_{x_{i-1}-x_i=2} \big( f(\vec{x}_{i}^+) - f(\vec{x})\big),$$
where we use the convention that the state $\vec{x}_{i}^+$ is obtained from $\vec{x}$ by incrementing by one the coordinate $x_i$.
\label{def:FTASEP}
\end{definition}
\begin{remark}
	One may similarly define FTASEP($\alpha$) on the state space $\mathbb{X}$ instead of  $\mathbb{X}_{>0}$, in order to allow initial conditions without a  rightmost particle. 
\end{remark}

In order to study FTASEP($\alpha$), we use a coupling with another interacting particle system: a TASEP with a source at the origin that injects particles at exponential rate $\alpha$. We consider configurations of particles on $\Z_{>0}$ where each site can be occupied by at most one particle, and each particle jumps to the right by one at exponential rate $1$, provided the target site is empty. At site $0$ sits an infinite source of particles, which means that a particle always jumps to site $1$ at exponential rate $\alpha$ when the site $1$ is empty (See Figure \ref{fig:OpenTASEP}). We will denote the occupation variables in half-line TASEP by $g_i(t)$ (equals $1$ if site $i$ is occupied, $0$ else). 
\begin{definition}
The half-line TASEP with open boundary condition is a continuous-time Markov process defined on the state space $\lbrace 0,1\rbrace^{\Z_{>0}}  $  via its Markov generator, acting on local functions $f: \lbrace 0,1\rbrace^{\Z_{>0}}  \to \R$ by
\begin{multline*} \mathcal{ L}^{half}_{\alpha} f(\vec{g}) = \alpha \big( f(1, g_2, g_3, \dots) - f(g_1, g_2, \dots)\big) \\ +  \sum_{x\in \Z_{>0}} g_x (1-g_{x+1}) \big( f(\vec{g}_{x, x+1}) - f(\vec{g})\big),
\end{multline*}
where the state $\vec{g}_{x, x+1}$ is obtained from $\vec{g}$ by exchanging occupation variables at sites $x$ and $x+1$.
We define the integrated current $N_x(t)$ as the number of particles on the right of site $x$ (or at site $x$) at time $t$.
\label{def:halfTASEP}
\end{definition}

\begin{figure}
	\centering 
\begin{tikzpicture}[scale=0.6]
\draw[thick] (-1.2, 0) circle(1.2);
\draw (-1.2,0) node{source};
\draw[thick] (0, 0) -- (12.5, 0);
\foreach \x in {1, ..., 12} {
	\draw (\x, 0.15) -- (\x, -0.15) node[anchor=north]{\footnotesize $\x$};
}
\fill[thick, gray] (2, 0) circle(0.2);
\fill[thick] (3, 0) circle(0.2);
\fill[thick, gray] (6, 0) circle(0.2);
\fill[thick] (7, 0) circle(0.2);
\fill[thick] (9, 0) circle(0.2);
\draw[thick, ->] (3, 0.3) to[bend left] (4, 0.3);
\draw[thick, ->] (7, 0.3) to[bend left] (8, 0.3);
\draw[thick, ->] (9, 0.3) to[bend left] (10, 0.3);
\draw[thick, ->] (-0.1, 0.3) to[bend left] (1, 0.3);
\end{tikzpicture}
\caption{Illustration of the half-line TASEP. The particles in gray cannot move because of the exclusion rule. }
\label{fig:OpenTASEP}
\end{figure}
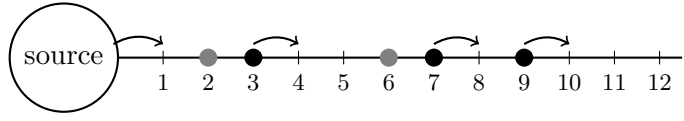

Define maps  
$$\begin{array}{lll}
\Phi_{>0} : & \mathbb{X}_{>0}  &\longrightarrow \lbrace 0,1\rbrace^{\Z_{>0}} \\ 
& (x_i)_{i\in \Z_{>0}} &\longmapsto \big( x_i-x_{i+1} -1 \big)_{i\in \Z_{>0}},
\end{array} 
$$
and 
$$\begin{array}{lll}
\Phi : & \mathbb{X}_{}  &\longrightarrow \lbrace 0,1\rbrace^{\Z_{}} \\ 
& (x_i)_{i\in \Z} &\longmapsto \big( x_i-x_{i+1} -1 \big)_{i\in \Z}.
\end{array} 
$$

\begin{proposition}
	Let $\vec{x}(t) = (x_n(t))_{n\geqslant 1}$ be the particles positions in the FTASEP($\alpha$) started from some initial condition $\vec{x}(0)\in\mathbb{X}_{>0}$ (resp. $\mathbb{X}_{}$). Then denoting $ \vec{g}(t) = \lbrace g_i(t)\rbrace_{\Z_{>0}} = \Phi(\vec{x}(t))$, the dynamics of  $\vec{g}(t)$ are those of half-line TASEP (resp. TASEP) starting from the initial configuration $\Phi_{>0}(\vec{x}(0))$ (resp.  $\Phi(\vec{x}(0))$).
	\label{prop:mapping}
\end{proposition}
\begin{proof}
We explain how the mapping between the two processes works in the half-space case (which corresponds to the FTASEP($\alpha$) defined on the space of configurations $\mathbb{X}_{>0}$), since this is the case we will be most interested in this paper, and the full space case is very similar. 

Assume that particles at positions $(x_i)_{i\in \Z_{>0}}$ follow the FTASEP($\alpha$) dynamics, starting from some initial condition $\vec{x}(0)\in\mathbb{X}_{>0}$,  and let us show that the $g_i = x_i-x_{i+1}-1$ follow the dynamics of the half-line TASEP occupation variables. 
	 
If $x_1=x_2+1$  (i.e. $g_1=0$), the first particle in FTASEP($\alpha$) jumps at rate $\alpha$. After it has jumped, $x_1=x_2+2$ (i.e. $g_1=1$). This corresponds to a particle arriving from the source to site $1$ in the half-line TASEP. After this jump, $x_1=x_2+2$ (i.e.  $g_1=1$), so that the first particle cannot move in the FTASEP($\alpha$) because of the facilitation rule and no particle can jump from the source in half-line TASEP. 
More generally, because of the exclusion and facilitation rules, the $(i+1)$th particle in FTASEP($\alpha$) can move only if $g_i=1$ and $g_{i+1}=0$ and does so at rate $1$. After the move, $x_{i+1}$ has increased by one so that $g_i=0$ and $g_{i+1}=1$. This exactly corresponds to the half-line TASEP dynamics. 
	 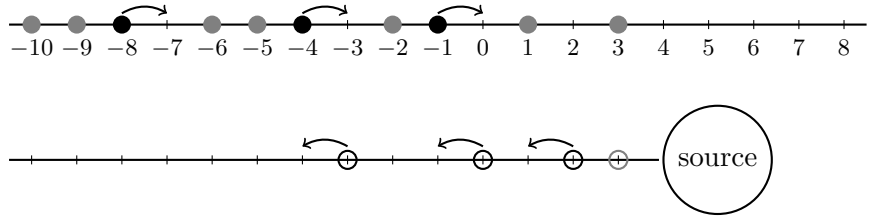
\begin{figure}
	 	\begin{tikzpicture}[scale=0.6]
	 	\draw[thick] (-9.5, 0) -- (9.5, 0);
	 	\foreach \x in {-10, ..., 8} {
	 		\draw (\x+1, 0.1) -- (\x+1, -0.1) node[anchor=north]{\footnotesize $\x$};
	 	}
	 	\fill[gray] (-9, 0) circle(0.2);
	 	\fill[gray] (-8, 0) circle(0.2);
	 	\fill (-7, 0) circle(0.2);
	 	\fill[gray] (-5, 0) circle(0.2);
	 	\fill[gray] (-4, 0) circle(0.2);
	 	\fill (-3, 0) circle(0.2);
	 	\fill[gray] (-1, 0) circle(0.2);
	 	\fill (0, 0) circle(0.2);
	 	\fill[gray] (2, 0) circle(0.2);
	 	\fill[gray] (4, 0) circle(0.2);
	 	\draw[thick, ->] (-7, 0.25) to[bend left] (-6, 0.25);
	 	\draw[thick, ->] (-3, 0.25) to[bend left] (-2, 0.25);
	 	\draw[thick, ->] (0, 0.25) to[bend left] (1, 0.25);
	 	\begin{scope}[yshift=-3cm]
	 	\draw[thick] (-9.5, 0) -- (4.9, 0);
	 	\foreach \x in {-9, ..., 4} {
	 		\draw (\x, 0.1) -- (\x, -0.1);
	 	}
	 	\draw[thick] (6.2, 0) circle(1.2);
	 	\draw (6.2,0) node{source};
	 	\draw[thick] (-2, 0) circle(0.2);
	 	\draw[thick] (1, 0) circle(0.2);
	 	\draw[thick] (3, 0) circle(0.2);
	 	\draw[thick, gray] (4, 0) circle(0.2);
	 	\draw[thick, ->] (-2, 0.3) to[bend right] (-3, 0.3);
	 	\draw[thick, ->] (1, 0.3) to[bend right] (0, 0.3);
	 	\draw[thick, ->] (3, 0.3) to[bend right] (2, 0.3);
	 	\end{scope}
	 	
	 	\end{tikzpicture}
	 	\caption{Illustration of the coupling. The dynamics of particles in the bottom picture is nothing else but the dynamics of the holes in the top picture. In order to see it more precisely, consider the holes in the top picture and shrink the distances so that the distance between two consecutive holes decreases by $1$; one gets exactly the bottom picture with the corresponding dynamics.}
	 	\label{fig:coupling}
	 \end{figure}
\end{proof}
\begin{remark}
	Formally, Proposition \ref{prop:mapping} means that for any $\vec{x} \in \mathbb{X}_{>0}$ and local function $f :\lbrace 0,1\rbrace^{\Z>0} \to \R$, 
	$$ \mathcal{ L}^{half}_{\alpha} f (\Phi(\vec x))  = \mathcal{ L}^{FTASEP}_{\alpha} \big(f \circ \Phi \big) (\vec x).$$ 
\end{remark}

In the following, we are mainly interested in FTASEP($\alpha$) starting from the step initial condition, or equivalently the half-line TASEP started from a configuration where all sites are initially empty. 
\begin{proposition}
Let $\vec{x}(t) =  (x_n(t) )\rbrace_{n\geqslant 1}$ be the particles positions in the FTASEP($\alpha$) started from step initial condition (see Definition \ref{def:FTASEP}).  Let $( N_x(t))_{x\in\Z_{>0}}$ be the currents in the half-line TASEP started from empty initial configuration (see Definition \ref{def:halfTASEP}). Then we have the equality in law of the processes
$$ \big( x_n(t)+n \big)_{n\geqslant 1, t\geqslant 0}  = \big( N_{n}(t)\big)_{n\geqslant 1, t\geqslant 0}.$$
\label{prop:coupling}
\end{proposition}
\begin{proof}
Because we start from step initial condition, $x_n(t)+n$ in FTASEP($\alpha$) equals the number of holes (empty sites) on the left of the $n$th particle.
Using Proposition \ref{prop:mapping}, and denoting the occupation variables in half-line TASEP by $g_i$, we have 
$$ x_n+n \overset{(d)}{=} \sum_{i\geqslant n} g_i  = N_n,$$
jointly for all $n$ as claimed. 
\end{proof}
Let us explain how Proposition \ref{prop:coupling} enables us to quickly recover the results from \cite{gabel2010facilitated}. We later provide rigorous results substantiating many of these claims, but for the moment just proceed heuristically. Consider the case $\alpha=1$.
One expects (and we prove in the next Section \ref{sec:FTASEPfluctuations}) that the law of large numbers for the current of particles in the half-line  TASEP is the same as in TASEP. Intuitively, this is because we expect that the law of large numbers is determined by a conservation PDE (of the form \eqref{eq:PDEconservation}) which is simply the restriction to a half-space of the conservation PDE governing the hydrodynamics of TASEP on the full line. Thus,
$$ \frac{N_{\kappa t}(t)}{t} \xrightarrow[t\to\infty]{a.s.} \frac{1}{4}(1-\kappa)^2. $$
Then, Proposition \ref{prop:coupling} implies that for FTASEP,
$$ \frac{ x_{\kappa t}}{t} \xrightarrow[t\to\infty]{a.s.} \frac{1}{4}(1-\kappa)^2 -\kappa = \frac{1-6\kappa + \kappa^2}{4}.$$
One can deduce the shape of the limiting density profile from the law of large numbers of particles positions. Let $\pi(\kappa)$ be the macroscopic position of the particle indexed by $\kappa t$, i.e.
$$ \pi(\kappa) =  \frac{1-6\kappa + \kappa^2}{4}. $$
This yields $ \kappa = 3-2\sqrt{2+\pi}$ (which can be interpreted as the limit of the integrated current in the FTASEP at site $\pi t$, rescaled by $t$). The density profile is obtained by differentiating $\kappa$ with respect to $\pi$, and we get (as in \cite[Equation (5)]{gabel2010facilitated})
$$ \rho(\pi t , t) = \frac{1}{\sqrt{2+\pi}}.$$

In light of the mapping between FTASEP and half-line TASEP from Proposition \ref{prop:mapping}, it is possible to write down a family of
translation invariant stationary measures in the FTASEP. They are given by choosing   gaps between consecutive particles as i.i.d Bernoulli random variables. From these, we may also deduce  the expression for the flux from \eqref{eq:flux}. Assume that the system is at equilibrium, such that the gaps between consecutive particles are i.i.d. and distributed according to the $Bernoulli(p)$ distribution. Let us call $\nu_{p}$ this measure on $\lbrace 0,1\rbrace^{\Z}$.  Then, by the renewal theorem, the average density $\rho$ is related to $p$ via
$$ \rho = \frac{1}{1+\EE[\mathrm{gap}]} = \frac{1}{1+p}.$$
The flux $j(\rho) $  is the product of the density times the drift of one particle, and since particles jump by $1$, the drift is given by the probability of a jump for a tagged particle, i.e. $p(1-p)$. Indeed, considering a tagged particle in the stationary distribution, then its right neighbour has a probability $p$ of being empty and its left neighbour has a probability $1-p$ of being occupied. This yields
\begin{equation}
j(\rho) = \rho (1-p) p  = \frac{(1-\rho)(2\rho-1)}{\rho}.
\label{eq:flux2}
\end{equation}

\subsection{Proofs of limit theorems}
We use now the coupling from Proposition \ref{prop:coupling} to translate the asymptotic results about last passage percolation from Theorems \ref{theo:LPPdiagointro}, \ref{theo:LPPawaydiagointro}, \ref{theo:crossfluctuations} and \ref{theo:SU}  into limit theorems for the FTASEP($\alpha$).
\label{sec:FTASEPfluctuations}

	Let $\vec{x}(t) = \lbrace x_n(t) \rbrace_{n\geqslant 1}$ be the particles positions in the FTASEP($\alpha$) started from step initial condition. Using Proposition \ref{prop:coupling}, we have that for any $y\in \R$
\begin{align*}
\PP\big( x_n(t)\leqslant y \big) &=\PP\big( x_n(t)\leqslant \lfloor y\rfloor \big)\\
&=\PP\big( N_n(t)\leqslant \lfloor y\rfloor +n \big).
\end{align*}
In order to connect the problem with half-space last passage percolation, we use the next result.
\begin{lemma}
Consider the exponential LPP model in a half-quadrant where the weights on the diagonal have parameter $\alpha$, and recall the definition of last passage times $H(n,m)$ from Definition \ref{def:LPPexp}. Consider the half-line TASEP  where the source injects particles at rate $\alpha$ with empty initial configuration and recall $N_x(t)$, the current at site $x$. Then for any $t>0$ and  $n, y \in \Z_{>0}$  we have that
$$ \PP\big( N_n(t)\leqslant y \big) =\PP\big(H(n+y-1, y)\geqslant t\big).$$
\label{lem:couplinglppTASEP}
\end{lemma}
\begin{proof}
This is due to a standard mapping \cite{rost1981non} between  exclusion processes and last passage percolation, where the border of the percolation cluster can be interpreted as a height function for the exclusion process. More precisely, the processes have to be coupled in such a way  that the weight $w_{ij}$ in the LPP model is the $(i-j+1)^{th}$ waiting time of the $j^{th}$ particle in the half-line TASEP -- the waiting time is  counted from the moment when it can jump, and by convention the first waiting time is when it jumps from the source into the system.
\end{proof}

\subsubsection{GSE ($\alpha>1/2$) and GOE ($\alpha=1/2$) cases:}
By Theorems \ref{theo:LPPdiagointro} we have that $$ H(n,n) = 4n + \sigma n^{1/3} \chi_n, $$
where
$ \sigma = 2^{4/3} $ and $\chi_n$ is a sequence of random variables weakly converging to the GSE (divided by $\sqrt{2}$ according to the convention chosen in Definition \ref{def:GSEdistribution}) distribution when $\alpha >1/2$ and to the GOE distribution when $\alpha=1/2$.
Let $y\in\R$ be fixed and $\varsigma>0$ be a coefficient to specify later. For $t>0$, we have
\begin{align*}
\PP\left(x_1(t) \leqslant \frac{t}{4} + t^{1/3}\varsigma y \right) &=  \PP\left(H\left(\Big\lfloor\frac{t}{4}+ t^{1/3}\varsigma  y \Big\rfloor, \Big\lfloor\frac{t}{4}+ t^{1/3}\varsigma y \Big\rfloor\right)\geqslant t\right)\\
&=  \PP\left(4\left(\Big\lfloor\frac{t}{4} + t^{1/3}\varsigma y \Big\rfloor \right) + \sigma \Big\lfloor\frac{t}{4} + t^{1/3}\varsigma y \Big\rfloor^{1/3} \chi_{\big\lfloor\frac{t}{4} + t^{1/3}\varsigma y \big\rfloor} \geqslant t\right)\\
&= \PP\left( 4 \varsigma y t^{1/3} + \big(\sigma (t/4)^{1/3}+ o(t^{1/3}\big) \chi_{\big\lfloor\frac{t}{4} + t^{1/3}\varsigma y \big\rfloor} \geqslant o(t^{1/3})\right)
\end{align*}
where the $o(t^{1/3})$ errors are deterministic. Thus, if we set $\varsigma = 2^{-4/3}$, we obtain that
$$ \lim_{t\to\infty}
\PP\left(x_1(t) \geqslant \frac{t}{4} + t^{1/3}\varsigma y \right) =  \lim_{t\to\infty}\PP\left(    \chi_{\big\lfloor\frac{t}{4} + t^{1/3}\varsigma y \big\rfloor} \leqslant -y\right) =  F_{\rm GSE}(-\sqrt{2} y)$$
 when $\alpha>1/2$ and
$$ \lim_{t\to\infty}
\PP\left(x_1(t) \geqslant \frac{t}{4} + t^{1/3}y \right) =  F_{\rm GOE}( -y)$$
when $\alpha=1/2$.

\subsubsection{Gaussian case:}
By Theorem \ref{theo:LPPdiagointro} we have that $$ H(n,n) = h(\alpha)n + \sigma n^{1/2} G_n, $$
where
$h(\alpha) = \frac{1}{\alpha(1-\alpha)}$,  $ \sigma = \frac{1-2\alpha}{\alpha^2(1-\alpha)^2}$ and $G_n$ is a sequence of random variables weakly converging to the standard Gaussian when $\alpha <1/2$.
As in the previous case, let $y\in\R$ be fixed and $\varsigma>0$ be a coefficient to specify  later. For $t>0$, we have
\begin{align*}
&\PP\left(x_1(t) \leqslant t/h(\alpha) + t^{1/2}\varsigma y \right)\\ 
& =  \PP\left(H\left(\Big\lfloor\frac{t}{h(\alpha)}+ t^{1/2}\varsigma  y \Big\rfloor, \Big\lfloor\frac{t}{h(\alpha)}+ t^{1/2}\varsigma y \Big\rfloor\right)\geqslant t\right)\\
&=  \PP\left(h(\alpha)\left(\Big\lfloor\frac{t}{h(\alpha)} + t^{1/2}\varsigma y \Big\rfloor \right) + \sigma \Big\lfloor\frac{t}{h(\alpha)} + t^{1/2}\varsigma y \Big\rfloor^{1/2} G_{\big\lfloor\frac{t}{h(\alpha)} + t^{1/2}\varsigma y \big\rfloor} \geqslant t\right)\\
&= \PP\left( h(\alpha) \varsigma y t^{1/2} + \sigma (t/h(\alpha))^{1/2}G_{\big\lfloor\frac{t}{4} + t^{1/3}\varsigma y \big\rfloor} \geqslant o(t^{1/2})\right).
\end{align*}
 Thus, if we set $\varsigma = \frac{1-2\alpha}{\sqrt{\alpha(1-\alpha)}}$, we obtain that
$$ \lim_{t\to\infty}
\PP\left(x_1(t) \geqslant \frac{t}{h(\alpha)} + t^{1/2}\varsigma y \right) =  \lim_{t\to\infty}\PP\left(    G_{\big\lfloor\frac{t}{h(\alpha)} + t^{1/2}\varsigma y \big\rfloor} \leqslant -y\right) =  G(-y).$$

\subsubsection{GUE case:}
We have
\begin{multline}
\PP\left(x_{\lfloor rt \rfloor}(t)  \leqslant \pi t + \varsigma t^{1/3}y \right) = \\
\PP\left(H\left( 2 \big\lfloor rt \big\rfloor + \big\lfloor  \pi t + \varsigma
 t^{1/3}y \big\rfloor - 1, \lfloor rt \big\rfloor + \big\lfloor  \pi t + \varsigma t^{1/3}y \big\rfloor  \right)\geqslant t\right).
\label{eq:grosseproba}
\end{multline}
By Theorem \ref{theo:LPPawaydiagointro} we have that for $m=\kappa n+ \mathcal{O}(n^{1/3})$,
$$ H(n,m) = (1+\sqrt{\kappa})^2 n + \sigma n^{1/3} \chi_n, $$
where
$$\sigma = \frac{(1+ \sqrt{\kappa})^{4/3}}{\sqrt{\kappa}^{1/3}}$$
and  $\chi_n$ is a sequence of random variables weakly converging to the GUE  distribution, for $\alpha >\frac{\sqrt{\kappa}}{1+\sqrt{\kappa}}$. Hence
\begin{equation*}
\eqref{eq:grosseproba}= \PP\left(  \left( 1+\sqrt{\frac{\pi}{r+\pi}}\right)^2 n + \sigma n^{1/3} \chi_n \geqslant t \right)
\end{equation*}
where $n=2\big\lfloor rt \big\rfloor + \big\lfloor  \pi t + \varsigma t^{1/3}y \big\rfloor - 1$. Choosing $\pi$ such that
$ \left( 1+\sqrt{\frac{r+ \pi}{2r+\pi}}\right)^2(2r+\pi)=1$,
i.e.
$$\pi= \frac{1-6r +r^2}{4},$$ we get that
$$ \eqref{eq:grosseproba} = \PP\left(\left( 1+\sqrt{\frac{r+\pi}{2r+\pi}}\right)^2 \varsigma t^{1/3} y + \sigma \big((2r+\pi)t \big)^{1/3} \chi_n \geqslant o(t^{1/3})\right).$$
Hence, letting
$$ \varsigma = \sigma (2r+\pi)^{4/3} =  2^{-4/3}\frac{(1+r)^{5/3}}{(1-r)^{1/3}},$$
yields
$$ \lim_{t\to\infty} \eqref{eq:grosseproba} =\lim_{n\to\infty} \PP\big(   \chi_n \geqslant -y + o(1)\big),$$
so that
$$ \lim_{t\to\infty}
\PP\left(x_{\lfloor rt \rfloor}(t) \geqslant \frac{(1-6r +r^2)t}{4} + t^{1/3}y\varsigma \right) =  F_{\rm GUE}(-y),$$
for $\alpha > \frac{1-r}{2}$ (This condition comes from the condition $\alpha > \frac{\sqrt{\kappa}}{1+\sqrt{\kappa}}$ in Theorem  \ref{theo:LPPawaydiagointro}).

\subsubsection{Crossover case:}
For the sake of clarity, we explain how the proof works in the one-point case. The multipoint case is similar. 
Assume 
$$ \alpha = \frac{1+2\sigma^{-1}\varpi n^{-1/3}}{2}.$$
Combining Lemma \ref{lem:couplinglppTASEP} and Proposition \ref{prop:coupling} as before, 
$$ \PP\left( H_n(\eta) \leqslant p \right)  = \PP\left( x_{2n^{2/3}\xi\eta+1}(4n+ (p\sigma- \xi^2\eta^2)n^{1/3}) \geqslant n- 3 n^{2/3}\xi \eta\right),$$
where $\xi=2^{2/3}$. 
Letting 
$$ t=4n+ (p\sigma- \xi^2\eta^2)n^{1/3},  $$
we have that 
$$  \alpha = \frac{1+2^{4/3}\varpi t^{-1/3}}{2} + o(t^{-1/3}),$$
$$ 2n^{2/3}\xi\eta+1 = 2^{1/3} \eta t^{2/3} + o(t^{1/3}),$$
$$  n- 3 n^{2/3}\xi = \frac{t}{4} - \frac{3 \eta\xi}{2 2^{1/3}}t^{2/3} + \frac{\eta^2 \xi^2 -\sigma  p}{2^{8/3}} t^{1/3} + o(t^{1/3}).$$
Hence, under this matching of parameters 
$$ \lim_{n \to \infty } \PP\left( H_n(\eta) \leqslant p \right) = \lim_{t\to\infty} \PP\left(  X_{t}(\eta)  \geqslant p  \right),$$
where the rescaled position $X_n(\eta)$ is defined in \eqref{eq:defrescaledposition}.

\begin{remark}
Although we do not attempt in this paper to make an exhaustive analysis of the FTASEP with respect to varying  initial conditions or parameters, such further analysis is allowed by our framework in several directions. In terms of initial condition, Proposition \ref{prop:mapping} allows to study the process starting form combinations of the wedge, flat or stationary initial data and translate to FTASEP some of the results known from TASEP \cite{prahofer2002current, arous2011current}. In terms of varying parameters, one could study the effect  of varying $\alpha$ or the speed of the next  few particles and one should observe the BBP transition \cite{baik2005phase} when considering fluctuations of $x_{r t}$ for $r>0$ (See Remark 1.6  in \cite{baik2017pfaffian}).
\label{rem:slowparticle}
\end{remark}

\section{Fredholm Pfaffian formulas for $k$-point distributions}
\label{sec:kpointdistribution}

We recall in this Section a result from \cite{baik2017pfaffian} which characterizes the joint probability distribution of passage times in the half-space exponential LPP model.  

\begin{proposition}[{\protect\cite[Proposition 1.7]{baik2017pfaffian}}]
	For any $h_1, \dots, h_k>0$ and integers $0<n_1 < n_2 < \dots <n_k$ and $m_1> m_2 > \dots > m_k$ such that $n_i>m_i$ for all $i$,  we have that
	$$ \mathbb{P}\big( H(n_1, m_1) < h_1, \dots,  H(n_k, m_k) < h_k \big)= \Pf\big(J-\kernel^{\rm exp}\big)_{\mathbb{L}^2\big( \mathbb{D}_k(h_1, \dots, h_k) \big)},$$
	where $J$ is the matrix kernel 
	\begin{equation}
	J(i, u ; j, v) = \mathds{1}_{(i, u)=(j,v)}\left(\begin{matrix}
	0 &1 \\
	-1 & 0
	\end{matrix} \right),
	\label{eq:kernelJ}
	\end{equation}
	and 
	$$ \mathbb{D}_k(g_1, \dots, g_k) = \lbrace (i,x)\in \lbrace 1, \dots, k\rbrace\times\R: x \geqslant g_i\rbrace.$$
	\label{prop:kernelexponential}
\end{proposition}
\label{sec:kernelexp}
The kernel $\kernel^{\rm exp}$ was introduced in \cite{baik2017pfaffian} in Section 4.4. It is defined on the state-space $(\lbrace 1, \dots, k\rbrace \times \R)^2$ and takes values in the space of skew-symmetric $2\times2 $ real matrices. The entries are given by 
$$ \kernel^{\rm exp}(i,x;j,y) = \Ikernel^{\rm exp}(i,x;j,y) + \begin{cases} \Rkernel^{\rm exp}(i,x;j,y) & \mbox{when }\alpha>1/2,\\
\hat{\Rkernel}^{\rm exp}(i,x;j,y) & \mbox{when }\alpha<1/2,\\
\bar{\Rkernel}^{\rm exp}(i,x;j,y) & \mbox{when }\alpha=1/2.\\  \end{cases}$$
Recalling the Definition \ref{def:basicrays} for integration contours in the complex plane, we define $\Ikernel^{\rm exp}$ by the following formulas.
\begin{multline*}
\Ikernel^{\rm exp}_{11}(i, x; j, y) :=  \int_{\mathcal{C}_{1/4}^{\pi/3}}\frac{\mathrm{d}z}{2\I\pi}\int_{\mathcal{C}_{1/4}^{\pi/3}}  \frac{\mathrm{d}w}{2\I\pi} \\  \frac{z-w}{4zw(z+w)} e^{-xz-yw}
\frac{(1+2z)^{n_i}(1+2w)^{n_j}}{(1-2z)^{m_i}(1-2w)^{m_j}} (2z+2\alpha -1)(2w+2\alpha -1).
\end{multline*}
\begin{multline}
\Ikernel^{\rm exp}_{12}(i, x; j, y) :=  \int_{\mathcal{C}_{a_z}^{\pi/3}}\frac{\mathrm{d}z}{2\I\pi}\int_{\mathcal{C}_{a_w}^{\pi/3}} \frac{\mathrm{d}w}{2\I\pi} \\  \frac{z-w}{2z(z+w)} e^{-xz-yw}
 \frac{(1+2z)^{n_i}}{(1-2w)^{n_j}}\frac{(1+2w)^{m_j}}{(1-2 z)^{m_i}}  \frac{2\alpha -1 + 2z}{2\alpha -1-2w} ,
 \label{eq:I12exp}
 \end{multline}
 where in the definition of the contours $\mathcal{C}_{a_z}^{\pi/3}$ and $ \mathcal{C}_{a_w}^{\pi/3} $, the constants $a_z, a_w\in\R$ are chosen so that $0<a_z<1/2$, $a_z+a_w >0$  and $a_w<(2\alpha-1)/2$.
\begin{multline}
\Ikernel^{\rm exp}_{22}(i, x; j, y) :=    \int_{\mathcal{C}_{b_z}^{\pi/3}}\frac{\mathrm{d}z}{2\I\pi}\int_{\mathcal{C}_{b_w}^{\pi/3}}\frac{\mathrm{d}w}{2\I\pi} \\ \frac{z-w}{z+w} e^{-xz-yw}
\frac{(1+2z)^{m_i}(1+2w)^{m_j}}{(1-2z)^{n_i}(1-2w)^{n_j}} \frac{1}{2\alpha -1 - 2z}\frac{1}{2\alpha -1 - 2w},
\label{eq:I22exp}
\end{multline}
where in the definition of the contours $\mathcal{C}_{b_z}^{\pi/3}$ and $ \mathcal{C}_{b_w}^{\pi/3} $, the constants $b_z, b_w\in\R$ are chosen so that  $0<b_z, b_w<(2\alpha-1)/2$ when $\alpha>1/2$, while we impose only $b_z, b_w>0$ when $\alpha\leqslant 1/2$.

We set $\Rkernel_{11}^{\rm exp}(i,x ; j,y) = 0,$ and $\Rkernel_{12}^{\rm exp}(i,x ; j,y)=0$ when $i\geqslant j$, and likewise for $\hat{R}^{\rm exp}$ and $\bar{R}^{\rm exp}$. The other entries  depend on the value of $\alpha$ and the sign of $x-y$.

\textbf{Case $\alpha>1/2$:} When $x>y$,
\begin{equation}\label{R22 new case1}
 \Rkernel_{22}^{\rm exp}(i,x ; j,y)  = - \int_{\mathcal{C}_{a_z}^{\pi/3}} \frac{\mathrm{d}z}{2\I\pi} \frac{(1+2z)^{m_i}(1-2z)^{m_j}}{(1-2z)^{n_i}(1+2z)^{n_j}} \frac{1}{2\alpha-1-2z}\frac{1}{2\alpha-1+2z} 2ze^{-\vert x-y\vert z},\end{equation}
and when $x<y$
$$ \Rkernel_{22}^{\rm exp}(i,x ; j,y)  = \int_{\mathcal{C}_{a_z}^{\pi/3}}\frac{\mathrm{d}z}{2\I\pi} \frac{(1+2z)^{m_j}(1-2z)^{m_i}}{(1-2z)^{n_j}(1+2z)^{n_i}} \frac{1}{2\alpha-1-2z}\frac{1}{2\alpha-1+2z} 2ze^{-\vert x-y\vert z},$$
where $(1-2\alpha)/2<a_z<(2\alpha-1)/2$. One immediately checks that $\Rkernel_{22}^{\rm exp}$ is antisymmetric as we expect. When $i<j$ and $x>y$
$$ \Rkernel_{12}^{\rm exp}(i,x ; j,y)  = - \int_{\mathcal{C}_{1/4}^{\pi/3}} \frac{\mathrm{d}z}{2\I\pi}
\frac{(1+2z)^{n_i}}{(1+2z)^{n_j}}\frac{(1-2z)^{m_j}}{(1-2z)^{m_i}}
e^{-\vert x-y\vert z},$$
while if $x<y$, $ \Rkernel_{12}^{\rm exp}(i,x ; j,y)=\Rkernel_{12}^{\rm exp}(i,y ; j,x)$. Note that $\Rkernel_{12}$ is not antisymmetric nor symmetric (except when $k=1$, i.e. for the one point distribution).

\textbf{Case $\alpha<1/2$:} When $x>y$, we have
\begin{multline}
 \hat{\Rkernel}_{22}^{\rm exp}(i,x ; j,y)  =\frac{-e^{\frac{1-2\alpha}{2}y}}{2}  \int \frac{\mathrm{d}z}{2\I\pi} \frac{(1+2z)^{m_i}(2\alpha)^{m_j}}{(1-2z)^{n_i}(2-2\alpha)^{n_j}} \frac{e^{-xz}}{2\alpha-1+2z} \\
 + \frac{e^{\frac{1-2\alpha}{2}x}}{2}  \int \frac{\mathrm{d}z}{2\I\pi} \frac{(1+2z)^{m_j}(2\alpha)^{m_i}}{(1-2z)^{n_j}(2-2\alpha)^{n_i}} \frac{e^{-yz}}{2\alpha-1+2z} \\
 - \int_{\mathcal{C}_{a_z}^{\pi/3}} \frac{\mathrm{d}z}{2\I\pi} \frac{(1+2z)^{m_i}(1-2z)^{m_j}}{(1-2z)^{n_i}(1+2z)^{n_j}} \frac{1}{2\alpha-1-2z}\frac{1}{2\alpha-1+2z} 2ze^{-\vert x-y\vert z} \\
-\frac{e^{(x-y)\frac{1-2\alpha}{2}}}{4}\frac{(2\alpha)^{m_i}(2-2\alpha)^{m_j}}{(2-2\alpha)^{n_i}(2\alpha)^{n_j}}+ \frac{e^{(y-x)\frac{1-2\alpha}{2}}}{4}\frac{(2\alpha)^{m_j}(2-2\alpha)^{m_i}}{(2-2\alpha)^{n_j}(2\alpha)^{n_i}},
\label{eq:R22expcasalphapetit1}
\end{multline}
where the contours in the two first integrals pass to the right of $ (1-2\alpha)/2 $. When $x<y$, the sign of the third term is flipped  so that $\hat{\Rkernel}_{22}^{\rm exp}(i,x ; j,y)=-\hat{\Rkernel}_{22}^{\rm exp}(j, y ; i,x)$.
One can write slightly simpler formulas by reincorporating residues in the first two integrals: thus, when $x>y$,
\begin{multline}
 \hat{\Rkernel}_{22}^{\rm exp}(i,x ; j,y)  =\frac{-e^{\frac{1-2\alpha}{2}y}}{2}  \int_{\mathcal{C}_{a_z}^{\pi/3}} \frac{\mathrm{d}z}{2\I\pi} \frac{(1+2z)^{m_i}(2\alpha)^{m_j}}{(1-2z)^{n_i}(2-2\alpha)^{n_j}} \frac{e^{-xz}}{2\alpha-1+2z} \\
 + \frac{e^{\frac{1-2\alpha}{2}x}}{2} \int_{\mathcal{C}_{a_z}^{\pi/3}} \frac{\mathrm{d}z}{2\I\pi} \frac{(1+2z)^{m_j}(2\alpha)^{m_i}}{(1-2z)^{n_j}(2-2\alpha)^{n_i}} \frac{e^{-yz}}{2\alpha-1+2z} \\
 - \int_{\mathcal{C}_{a_z}^{\pi/3}} \frac{\mathrm{d}z}{2\I\pi} \frac{(1+2z)^{m_i}(1-2z)^{m_j}}{(1-2z)^{n_i}(1+2z)^{n_j}} \frac{1}{2\alpha-1-2z}\frac{1}{2\alpha-1+2z} 2ze^{-\vert x-y\vert z},
\end{multline}
where $\frac{2\alpha-1}{2} <a_z<\frac{1-2\alpha}{2}$.
When $i<j$, if $x>y$
\begin{equation}
 \hat{\Rkernel}_{12}^{\rm exp}(i,x ; j,y)  = - \int_{\mathcal{C}_{1/4}^{\pi/3}}\frac{\mathrm{d}z}{2\I\pi}\frac{(1+2z)^{n_i}}{(1+2z)^{n_j}}\frac{(1-2z)^{m_j}}{(1-2z)^{m_i}}e^{-\vert x-y\vert z},
 \label{eq:R12expcasalphapetit1}
\end{equation}
while if $x<y$, $ \hat{\Rkernel}_{12}^{\rm exp}(i,x ; j,y)=\hat{\Rkernel}_{12}^{\rm exp}(i,y ; j,x)$.

\textbf{Case $\alpha=1/2$:} When $x>y$,
\begin{multline}\label{R22 new case3}
 \bar{\Rkernel}_{22}^{\rm exp}(i,x ; j,y)  = - \int_{\mathcal{C}_{1/4}^{\pi/3}}\frac{\mathrm{d}z}{2\I\pi} \frac{(1+2z)^{m_i}}{(1-2z)^{n_i}} \frac{e^{-xz}}{4z}
 +  \int_{\mathcal{C}_{1/4}^{\pi/3}}\frac{\mathrm{d}z}{2\I\pi} \frac{(1+2z)^{m_j}}{(1-2z)^{n_j}} \frac{e^{-yz}}{4z} \\
 +\int_{\mathcal{C}_{1/4}^{\pi/3}} \frac{\mathrm{d}z}{2\I\pi} \frac{(1+2z)^{m_i}(1-2z)^{m_j}}{(1-2z)^{n_i}(1+2z)^{n_j}} \frac{e^{-\vert x-y\vert z}}{2z} \ \  -\frac{1}{4},\\
\end{multline}
with a modification of the last two terms when $x<y$ so that $\bar{\Rkernel}_{22}^{\rm exp}(i,x ; j,y)=-\bar{\Rkernel}_{22}^{\rm exp}(j, y ; i,x)$.
When $i<j$, if $x>y$
\begin{equation*}
 \bar{\Rkernel}_{12}^{\rm exp}(i,x ; j,y)  = - \int_{\mathcal{C}_{1/4}^{\pi/3}} \frac{\mathrm{d}z}{2\I\pi}\frac{(1+2z)^{n_i}}{(1+2z)^{n_j}}\frac{(1-2z)^{m_j}}{(1-2z)^{m_i}}e^{-\vert x-y\vert z},
\end{equation*}
while if
$x<y$, $ \bar{\Rkernel}_{12}^{\rm exp}(i,x ; j,y)=\bar{\Rkernel}_{12}^{\rm exp}(i,y ; j,x)$.

\begin{remark}
	It may be possible to write simpler integral formulas for $\kernel^{\rm exp}$ by changing the contours used in the definition of $\Ikernel^{\rm exp}$ and identifying certain terms of $\Rkernel^{\rm exp}$ as residues of the integrand in  $\Ikernel^{\rm exp}$. The reason why we have written the kernel  $\Ikernel^{\rm exp}$ as above is mostly technical. For the asymptotic analysis of these formulas, it is convenient that  all  contours may be deformed so that they approach $0$ without encountering any singularity, as will be explained in Section \ref{sec:crossoverasymptotics}. 
\end{remark}

\section{Asymptotic analysis in the crossover regime}
\label{sec:crossoverasymptotics}
This section is devoted to the proofs of Theorems \ref{theo:crossfluctuations} and \ref{theo:SU}. We start by providing formulas for the correlation kernels $\kernel^{\rm cross}$ and $\kernel^{\rm SU}$ used in the statements of both theorems. 

The published version of this paper contains several mistakes in the definition and derivation of the correlation kernels below that resulted in some incorrect choices of contours and some missing terms in the kernels. This was pointed out to us by Zongrui Yang, prompted by his joint work with Evgeni Dimitrov \cite{dimitrov-yang2024preparation}. Yang kindly provided suggested edits to resolve those issues, and we have included those below. In particular, our earlier version of this paper had issues when the parameter $\varpi$ was less than or equal to $0$. We graciously acknowledge Yang's help in this revision.

\subsection{Formulas for $\kernel^{\rm cross}$}
\label{sec:formulaKcross}
The kernel $\kernel^{\rm cross}$ introduced  in  \cite[Section 2.5]{baik2017pfaffian} can be written  as 
 	$$ \kernel^{\rm cross}(i, x;j, y) = \Ikernel^{\rm cross}(i, x;j, y)+\Rkernel^{\rm cross}(i, x;j, y),$$
 	where $ \kernel^{\rm cross}_{21}(i, x;j, y)=-\kernel^{\rm cross}_{12}(j, y; i, x)$, and   we have
 	\begin{align*}
 	\Ikernel_{11}^{\rm cross}(i, x;j, y) =&  \int_{\mathcal{C}_{1}^{\pi/3}} \frac{\mathrm{d}z}{2\I\pi}\int_{\mathcal{C}_{1}^{\pi/3}}\frac{\mathrm{d}w}{2\I\pi} \\ 
 	&\frac{z+\eta_i -w-\eta_j}{z+w+ \eta_i+\eta_j} \frac{z+\varpi+\eta_i}{z+\eta_i}\frac{w+\varpi+\eta_j}{w+\eta_j}  e^{z^3/3 + w^3/3 - x z -y w},\\
 	\Ikernel_{12}^{\rm cross}(i, x;j, y) =& \int_{\mathcal{C}_{a_z}^{\pi/3}} \frac{\mathrm{d}z}{2\I\pi}\int_{\mathcal{C}_{a_w}^{\pi/3}} \frac{\mathrm{d}w}{2\I\pi} \\ 
 	&\frac{z +\eta_i -w+\eta_j  }{2(z+\eta_i)(z+\eta_i+w-\eta_j)}\frac{z+\varpi+\eta_i}{-w+\varpi+\eta_j}  e^{z^3/3 + w^3/3 - x z -yw} ,\\
 	\Ikernel_{22}^{\rm cross}(i, x;j, y) =&  \int_{\mathcal{C}_{b_z}^{\pi/3}} \frac{\mathrm{d}z}{2\I\pi}\int_{\mathcal{C}_{b_w}^{\pi/3}} \frac{\mathrm{d}w}{2\I\pi}
 	\frac{z-\eta_i-w+\eta_j}{4(z-\eta_i+w-\eta_j)}\frac{e^{z^3/3 + w^3/3 - x z-yw }}{(z-\varpi-\eta_i)(w-\varpi-\eta_j)}.
 	\end{align*}
 	The contours in $\Ikernel_{12}^{\rm cross}$ are chosen so that $a_z>-\eta_i$, $a_z+a_w>\eta_j-\eta_i$ and $a_w<\varpi+ \eta_j$.
 	The contours in $\Ikernel_{22}^{\rm cross}$ are chosen so that  (1) if $\varpi\leqslant0$ then $b_z>\eta_i$ and $b_w>\eta_j$, and (2) if $\varpi>0$ then $b_z\in(\eta_i,\eta_i+\varpi)$ and $b_w\in(\eta_j,\eta_j+\varpi)$. 
 	
 	We have $\Rkernel_{11}^{\rm cross}(i, x;j, y)=0$, and $\Rkernel_{12}^{\rm cross}(i, x;j, y)=0$ when $i\geqslant j$. When $i<j$,
 	$$ \Rkernel_{12}^{\rm cross}(i, x;j, y) = \frac{-\exp\left(\frac{-(\eta_i-\eta_j)^4+ 6(x+y)(\eta_i-\eta_j)^2+3(x-y)^2}{12(\eta_i-\eta_j)}\right)}{\sqrt{4\pi(\eta_j-\eta_i)}}, $$
 	which may also be written as 
 	$$\Rkernel_{12}^{\rm cross}(i, x;j, y)= -\int_{-\infty}^{+\infty} \mathrm{d}\lambda e^{-\lambda(\eta_i - \eta_j)} \Ai(x_i+\lambda)\Ai(x_j+\lambda).$$
 	The kernel  $\Rkernel_{22}^{\rm cross}$ is antisymmetric, and when $ x -   \eta_i^2>y -  \eta_j^2$ we have
 	\begin{multline}\label{cross R22}
 	\Rkernel_{22}^{\rm cross}(i, x;j, y)= 
 	 \frac{\mathds{1}_{\varpi\leqslant0}}{4} \int_{\mathcal{C}_{c_z}^{\pi/3}}\frac{\mathrm{d}z}{2\I\pi} \frac{\exp\big( (z+\eta_j)^3/3 +(\varpi+\eta_i)^3/3 -y (z+\eta_j) -x(\varpi+\eta_i)\big)}{\varpi+z} \\
 	 -\frac{\mathds{1}_{\varpi\leqslant0} }{4} \int_{\mathcal{C}_{c_z}^{\pi/3}}\frac{\mathrm{d}z}{2\I\pi} \frac{\exp\big( (z+\eta_i)^3/3 +(\varpi+\eta_j)^3/3 -x (z+\eta_i) -y(\varpi+\eta_j)\big)}{\varpi+z} \\
 	-\frac{1}{2}\int_{\mathcal{C}_{d_z}^{\pi/3}}\frac{\mathrm{d}z}{2\I\pi}\frac{z\exp\big( (z+\eta_i)^3/3 +(-z+\eta_j)^3/3 -x (z+\eta_i) -y(-z+\eta_j)\big)}{(\varpi+z)(\varpi -z)}\\
          -\frac{\mathds{1}_{\varpi=0}}{4} \exp\left(\eta_i^3/3+\eta_j^3/3-\eta_ix-\eta_jy\right),
 	\end{multline}
	where the contours are chosen so that $c_z<-\varpi$ and  (1) if $\varpi\neq0$ then $d_z$ is between $-\varpi$ and $\varpi$, and (2) if $\varpi=0$ then $d_z>0$.  
\subsection{Formulas for $\kernel^{\rm SU}$}
\label{sec:formulaKSU}
The kernel $\kernel^{\rm SU}$ introduced in \cite[Section 2.5]{baik2017pfaffian} decomposes as 
 	$$ \kernel^{\rm SU}(i, x;j, y) = \Ikernel^{\rm SU}(i, x;j, y)+\Rkernel^{\rm SU}(i, x;j, y),$$
 	where  $ \kernel^{\rm SU}_{21}(i, x;j, y)=-\kernel^{\rm SU}_{12}(j, y; i, x)$, and  we have
 	\begin{align*}
 	\Ikernel^{\rm  SU}_{11}(i, x;j, y) &=   \int_{\mathcal{C}_{1}^{\pi/3}} \frac{\mathrm{d}z}{2\I\pi}\int_{\mathcal{C}_{1}^{\pi/3}} \frac{\mathrm{d}w}{2\I\pi}
 	\frac{(z+\eta_i-w-\eta_j)e^{z^3/3 + w^3/3 - xz -yw}}{4(z+\eta_i)(w+\eta_j)(z+w+ \eta_i+\eta_j)}   , \\
 	\Ikernel^{\rm  SU}_{12}(i, x;j, y) &=  \int_{\mathcal{C}_{a_z}^{\pi/3}} \frac{\mathrm{d}z}{2\I\pi}\int_{\mathcal{C}_{a_w}^{\pi/3}} \frac{\mathrm{d}w}{2\I\pi}
 	\frac{z + \eta_i -w + \eta_j}{2(z+\eta_i)(z+w+\eta_i-\eta_j)}    e^{z^3/3 + w^3/3 - xz -yw} ,\\
 	\Ikernel^{\rm  SU}_{22}(i, x;j, y)&=  \int_{\mathcal{C}_{b_z}^{\pi/3}} \frac{\mathrm{d}z}{2\I\pi}\int_{\mathcal{C}_{b_w}^{\pi/3}}\frac{\mathrm{d}w}{2\I\pi}
 	\frac{z-\eta_i-w+\eta_j}{z-\eta_i+w-\eta_j}  e^{z^3/3 + w^3/3 - xz-yw}.
 	\end{align*}
 	The contours in $\Ikernel_{12}^{\rm SU}$ are chosen so that $a_z>-\eta_i$,  $a_z+a_w>\eta_j-\eta_i$.  
 	The contours in $\Ikernel_{22}^{\rm SU}$ are chosen so that $b_z>\eta_i$ and $b_w>\eta_j$.
 	
 	We have $\Rkernel_{11}^{\rm SU}(i, x;j, y)=0$, and $\Rkernel_{12}^{\rm SU}(i, x;j, y)=0$ when $i\geqslant j$.
 	When $i<j$,
 	$$ \Rkernel_{12}^{\rm SU}(i, x;j, y) = \Rkernel_{12}^{\rm cross}(i, x;j, y) = \frac{-\exp\left(\frac{-(\eta_i-\eta_j)^4+ 6(x+y)(\eta_i-\eta_j)^2+3(x-y)^2}{12(\eta_i-\eta_j)}\right)}{\sqrt{4\pi(\eta_j-\eta_i)}}.$$
 	The kernel  $\Rkernel_{22}^{\rm SU}$ is antisymmetric, and when $ x - \eta_i^2> y - \eta_j^2$ we have
 	\begin{multline*}
 	\Rkernel_{22}^{\rm SU}(i, x;j, y)= \\ 
 	-\frac{1}{2} \int_{\mathcal{C}_{0}^{\pi/3}}\frac{\mathrm{d}z}{2\I\pi} z\exp\big( (z+\eta_i)^3/3 +(-z+\eta_j)^3/3 -x (z+\eta_i) -y(-z+\eta_j)\big)
 	\end{multline*}
 	where the contours are chosen so that $a_z>-\varpi$ and $b_z$ is between $-\varpi$ and $\varpi$.

\subsection{Proof of Theorem \ref{theo:crossfluctuations}}

Recall that we scale $\alpha$ as
	$$ \alpha = \frac{1+2\sigma^{-1}\varpi n^{-1/3}}{2}.$$
The proof of Theorem \ref{theo:crossfluctuations} follows the same lines as that of Theorems 1.4 and 1.5 in Sections 5 and 6 of \cite{baik2017pfaffian} (corresponding to Theorems \ref{theo:LPPdiagointro} and \ref{theo:LPPawaydiagointro}  in the present paper). We introduce the rescaled  correlation kernel
	\begin{multline*}
	\kernel^{\mathrm{exp}, n}(i,  x_i ; j, x_j):= \\ 
	{\footnotesize \begin{pmatrix}
	\sigma^2 n^{2/3}e^{ \eta_i x_i +  \eta_j x_j - \eta_i^3/3-\eta_j^3/3}\kernel_{11}^{\rm exp}\big(i,X_i; j, X_j\big) &\ \  \sigma n^{1/3}e^{ \eta_i x_i -  \eta_j x_j - \eta_i^3/3+\eta_j^3/3} \kernel_{12}^{\rm exp}\big(i,X_i; j, X_j\big)\\
	\sigma n^{1/3}e^{ -\eta_i x_i +  \eta_j x_j+ \eta_i^3/3-\eta_j^3/3} \kernel_{21}^{\rm exp}\big(i,X_i; j, X_j\big) &\ \ e^{ -\eta_i x_i -  \eta_j x_j+ \eta_i^3/3+\eta_j^3/3}\kernel_{22}^{\rm exp}\big(i,X_i; j, X_j\big)
	\end{pmatrix}},\end{multline*}
	where
	$$ X_i = 4 n+n^{1/3}(\sigma x_i -\xi^2 \eta_i^2 ), $$
	so that we have
	$$\PP\left(H_n(\eta_1) < x_1, \dots , H_n(\eta_k) < x_k \right)  = \Pf\big( J- \kernel^{\rm  exp, n}\big)_{\mathbb{L}^2(\mathbb{D}_k(x_1, \dots, x_k))},$$
	where the quantity $H_n(\eta)$ is defined in Section \ref{sec:halfspaceLPP}. We will decompose the kernel as $\kernel^{\mathrm{exp}, n}(i,  x_i ; j, x_j)=\Ikernel^{\mathrm{exp}, n}(i,  x_i ; j, x_j)+\Rkernel^{\mathrm{exp}, n}(i,  x_i ; j, x_j)$ according to the formulas in  Section \ref{sec:kernelexp}. The parameter $\alpha$ can be greater or smaller than $1/2$ depending on the sign of $\varpi$, so that we will need to be careful with the choice of contours. 
	
In order to prove Theorem \ref{theo:crossfluctuations}, we need to show that 
\begin{equation}
\lim_{n\to\infty}\Pf\big( J- \kernel^{\rm  exp, n}\big)_{\mathbb{L}^2(\mathbb{D}_k(x_1, \dots, x_k))}  =  \Pf\big( J- \kernel^{\rm  cross}\big)_{\mathbb{L}^2(\mathbb{D}_k(x_1, \dots, x_k))}.
\label{eq:whatwewant}
\end{equation}
We will first show that the kernel 	$\kernel^{\rm  exp, n}(i,x; j,y)$ converges to $\kernel^{\rm  cross}(i,x; j,y)$ for fixed $(i,x; j,y)$. Then, we will prove uniform bounds on the kernel  $\kernel^{\rm  exp, n}$ so that the Fredholm Pfaffian is an absolutely convergent series of integrals and hence the pointwise convergence of kernels implies the convergence of Fredholm Pfaffians. 
	
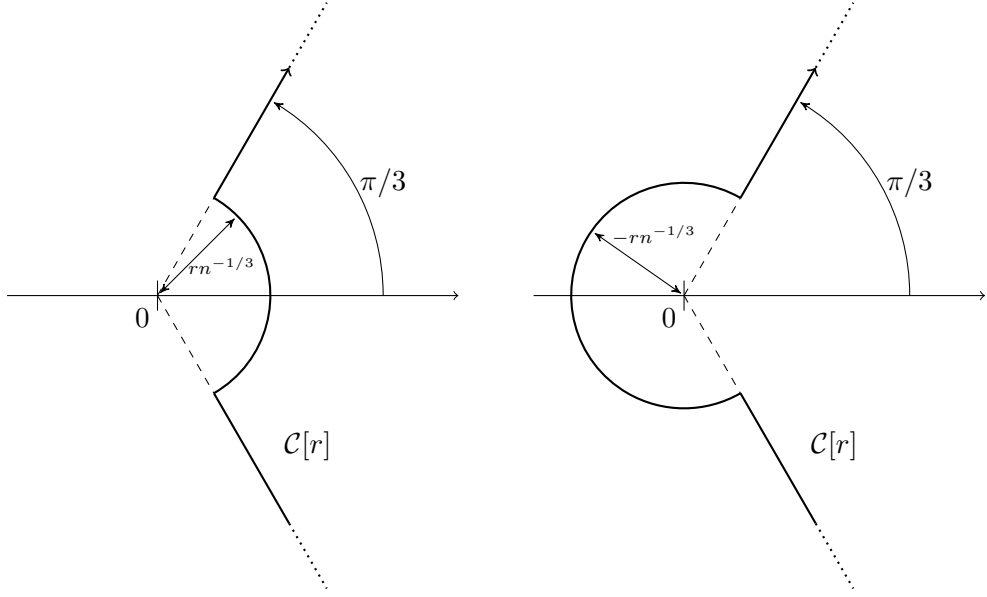
\begin{figure}
	\begin{tikzpicture}[scale=0.85]
	\draw[->] (-2,0) -- (4,0);
	\draw (0,-0.2) -- (0,0.2);
	\draw[thick, ->] (60:1.5) -- (60:3.5);
	\draw[dashed] (60:0) -- (60:1.5);
	\draw[thick] (-60:1.5) -- (-60:3.5);
	\draw[dashed] (-60:0) -- (-60:1.5);
	\draw[thick] (-60:1.5 ) arc(-60:60:1.5);
	\draw[->, >=stealth'] (3,0) arc(0:59:3);
	\draw node at (3,1.5) {$\pi/3$};
	\draw node at (2,-2) {$\mathcal{C}[r]$};
	\draw[<->, >=stealth'] (45:0.05) -- (45:1.45);
	\draw node at (0.85, 0.4) {\tiny{$r\  n^{-1/3}$}};
	\draw[thick, dotted] (60:3.5) -- (60:4.5);
	\draw[thick, dotted] (-60:3.5) -- (-60:4.5);
	\draw(-0.2, -0.3) node{$0$};

	 \begin{scope}[xshift=7cm]
	 \draw[->] (-2,0) -- (4,0);
	 \draw (0,-0.2) -- (0,0.2);
	 \draw[thick, ->] (60:1.5) -- (60:3.5);
	 \draw[dashed] (60:0) -- (60:1.5);
	 \draw[thick] (-60:1.5) -- (-60:3.5);
	 \draw[dashed] (-60:0) -- (-60:1.5);
	 \draw[thick] (60:1.5 ) arc(60:300:1.5);
	 \draw[->, >=stealth'] (3,0) arc(0:59:3);
	 \draw node at (3,1.5) {$\pi/3$};
	 \draw node at (2,-2) {$\mathcal{C}[r]$};
	 \draw[<->, >=stealth'] (145:0.05) -- (145:1.45);
	 \draw node at (-0.4, 0.8) {\tiny{$-r\  n^{-1/3}$}};
	 \draw[thick, dotted] (60:3.5) -- (60:4.5);
	 \draw[thick, dotted] (-60:3.5) -- (-60:4.5);
	 \draw(-0.2, -0.3) node{$0$};
	 \end{scope}
	\end{tikzpicture}
	\caption{The contours  $\mathcal{C}[r]$ when $r>0$ (left) and $\mathcal{C}[r]$ when $r<0$  (right).}
	\label{fig:contours}
\end{figure}

We introduce two types of modifications of the contour $\mathcal{C}_0^{\pi/3}$. For a parameter $r>0$, we denote by  $\mathcal{C}[r]$ the contour formed by the union of an arc of circle around $0$ of radius $r n^{-1/3}$, between $-\pi/3$ and $\pi/3$,  and two semi-infinite rays in directions $\pm\pi/3$ that connect the extremities of the arc to $\infty$ (see Figure \ref{fig:contours}, left). With this definition $0$ is on the left of the contour $\mathcal{C}[r]$.  For a parameter $r<0$, we denote by  $\mathcal{C}[r]$ a similar contour where the arc of circle has radius $-r$ and is now between angles from $\pi/3$ to $5\pi/3$ so that $0$ is on the right of $\mathcal{C}[r]$ (see Figure \ref{fig:contours}, right). 

Thanks to Cauchy's theorem, we have some freedom to deform the contours used in the definition of $\kernel^{\rm  exp}$ in Section \ref{sec:kernelexp}, as long as we do not cross any pole. Thus we can write   
	\begin{multline}
	\kernel^{\mathrm{exp}, n}_{11}(i,  x ; j, y) = e^{ \eta_i x +  \eta_j y- \eta_i^3/3-\eta_j^3/3} \sigma^2 n^{2/3} \int_{\mathcal{C}[1]}\frac{\mathrm{d}z}{2\I\pi} \int_{\mathcal{C}[1]} \frac{\mathrm{d}w}{2\I\pi} \frac{z-w}{4zw(z+w)}\\ (2z+2\sigma^{-1} \varpi n^{-1/3})(2w+2\sigma^{-1} \varpi n^{-1/3})
	\exp\Big(n(f(z)+f(w)) \\+ n^{2/3}(\xi\eta_i \log(1-4z^2)+ \xi\eta_j \log(1-4w^2)) \\ 
	+ n^{1/3} \xi^2 \eta_i^2 z + n^{1/3} \xi^2 \eta_j^2 w - n^{1/3}\sigma (x z +y w)\Big),
	\label{eq:kernelK11_modifiedcontour_cross}
	\end{multline}
	where the function $f$ is 
	$$ f(z) = -4 z + \log(1+2z) -\log(1-2z). $$
To take asymptotics of this expression, we use Laplace's method. The function $f$ has a double critical point at $0$.  We have 
\begin{equation}
f(z) = \frac{\sigma^3}{3} z^3 + \mathcal{O}(z^4),
\label{eq:Taylorapprox1}
\end{equation}
where $\sigma= 2^{4/3}$ and we know from Lemma 5.9 in \cite{baik2017pfaffian} that the contour $\mathcal{C}_{0}^{\pi/3}$ is steep-descent for $\Real[f]$ (which shows that the main contribution to the integral comes from integration in a neighborhood of $0$, see the proof of Theorem \ref{theo:LPPdiagointro}  in Section 5 of \cite{baik2017pfaffian}). 
Let us make  the change of variables $z= n^{-1/3}\tilde{z}/\sigma$ and likewise for $w$, and  use Taylor expansions of all terms in the integrand. Using the same kind of estimates (to control the error made when approximating the integrand) as in Proposition 5.8 in \cite{baik2017pfaffian}, we arrive at
	\begin{multline*}
	\kernel^{\mathrm{exp}, n}_{11}(i,  x ; j, y) \xrightarrow[n\to\infty]{} e^{ \eta_i x +  \eta_j y- \eta_i^3/3-\eta_j^3/3}   \int_{\mathcal{C}_{1}^{\pi/3}} \frac{\mathrm{d}z}{2\I\pi}\int_{\mathcal{C}_{1}^{\pi/3}} \frac{\mathrm{d}w}{2\I\pi}
	\frac{z-w}{z+w}\frac{z+\varpi}{z}\frac{w+\varpi}{w}  \\ \exp\left(z^3/3 + w^3/3 -4\xi\eta_i z^2/\sigma^2 -4\xi\eta_j w^2/\sigma^2 +\xi^2 \eta_i^2 z/\sigma +\xi^2\eta_j^2  w/\sigma- x z -  y w \right).
	\end{multline*}
	With our choice of $\sigma$ and $\xi$, we have that $4\xi/\sigma^2 = \xi^2/\sigma = 1$, so that after a change of variables (a simple translation where $z$ becomes $z+\eta_i$ and $w$ becomes $w+\eta_j$),
	\begin{multline*}
    \kernel^{\mathrm{exp}, n}_{11}(i,  x ; j, y) \xrightarrow[n\to\infty]{} \kernel^{\rm cross}_{11}(i,  x ; j, y) = \\  \int_{\mathcal{C}_{1}^{\pi/3}} \frac{\mathrm{d}z}{2\I\pi}\int_{\mathcal{C}_{1}^{\pi/3}} \frac{\mathrm{d}w}{2\I\pi}
	\frac{z+\eta_i-w-\eta_j}{ z+\eta_i+w+\eta_j  } \frac{z+\varpi+\eta_i}{z+\eta_i}\frac{w+\varpi+\eta_j}{w+\eta_j}  e^{z^3/3 + w^3/3 - x z -y w}.
	\end{multline*}
	
    Regarding $\kernel_{12}$, we write $\kernel^{\mathrm{exp}, n}_{12} = \Ikernel^{\mathrm{exp}, n}_{12} + \Rkernel^{\mathrm{exp}, n}_{12}$ where 
	\begin{multline}
	\Ikernel^{\mathrm{exp}, n}_{12}(i,  x ; j, y) = e^{ \eta_i x - \eta_j y-\eta_i^3/3+\eta_j^3/3}\sigma n^{1/3} \int_{\mathcal{C}[a_z]}\frac{\mathrm{d}z}{2\I\pi} \int_{\mathcal{C}[a_w]} \frac{\mathrm{d}w}{2\I\pi} \frac{z-w}{2z(z+w)} \\  \frac{2z+2\sigma^{-1} \varpi n^{-1/3}}{-2w+2\sigma^{-1} \varpi n^{-1/3}}
	\exp\Big(n(f(z)+f(w))  \\ + n^{2/3}(\xi\eta_i \log(1-4z^2)- \xi\eta_j \log(1-4w^2)) 
	+ n^{1/3} \xi^2 \eta_i^2 z \\ + n^{1/3} \xi^2 \eta_j^2 w - n^{1/3}\sigma (x z +y w)\Big),
	\label{eq:kernelK12_modifiedcontour_cross}
	\end{multline}
	where the contours are chosen so that $a_z>0$, $a_z+a_w>0$ and $a_w< \sigma^{-1} \varpi$.
	 Applying Laplace method as for $\kernel_{11}$, we  arrive at 
	\begin{multline*}
	\Ikernel^{\mathrm{exp}, n}_{12}(i,  x ; j, y) \xrightarrow[n\to\infty]{}  e^{ \eta_i x - \eta_j y-\eta_i^3/3+\eta_j^3/3} \int_{\mathcal{C}_{a_z}^{\pi/3}} \frac{\mathrm{d}z}{2\I\pi} \int_{\mathcal{C}_{a_w}^{\pi/3}} \frac{\mathrm{d}w}{2\I\pi}
	\frac{z-w}{2z(z+w)} \frac{z+\varpi}{-w+\varpi}\\ \exp\left(z^3/3 + w^3/3 -4\xi\eta_i z^2/\sigma^2 +4\xi\eta_j w^2/\sigma^2 +\xi^2\eta_i^2 z/\sigma +\xi^2\eta_j^2  w/\sigma - x z -  y w \right).
	\end{multline*}
 Thus, we find that after a change of variables 
	\begin{multline*}
	\Ikernel^{\mathrm{exp}, n}_{12}(i,  x ; j, y) \xrightarrow[n\to\infty]{}  \Ikernel^{\rm cross}_{12}(i,  x ; j, y) = \\  \int_{\mathcal{C}_{a_z}^{\pi/3}}\frac{\mathrm{d}z}{2\I\pi} \int_{\mathcal{C}_{a_w}^{\pi/3}} \frac{\mathrm{d}w}{2\I\pi}
	\frac{z + \eta_i  -w + \eta_j }{2(z+\eta_i)(z+\eta_i+w-\eta_j)}
	\frac{z+\varpi+\eta_i}{-w+\varpi+\eta_j}  e^{z^3/3 + w^3/3 - x z -y w},
	\end{multline*}
	where the contours in the last equation are now chosen so that $a_z>-\eta_i$, $a_z+a_w>\eta_j-\eta_i$ and $a_w<\varpi+ \eta_j$.
	When $i<j$ (and consequently $\eta_i<\eta_j$), and for $x,y $ such that  $\sigma x - \xi^2  \eta_i^2>\sigma y - \xi^2\eta_j^2$ (which is equivalent to $ x -   \eta_i^2> y -  \eta_j^2$), we use equation \eqref{eq:R12expcasalphapetit1} for $\Rkernel_{12}$ and find
	$$ \Rkernel_{12}^{\rm exp, n}(i,  x ; j, y) \xrightarrow[n\to\infty]{} -  \int_{\mathcal{C}_{1/4}^{\pi/3}} \frac{\mathrm{d}z}{2\I\pi} 
	\exp\big(   (z-\eta_i)^3/3 - (z-\eta_j)^3/3-x(z-\eta_i) +y(z-\eta_j)\big).
	$$
	One can check that with $ x -   \eta_i^2> y -  \eta_j^2$, the integrand is integrable on the contour $\mathcal{C}_{1/4}^{\pi/3}$. When $x - \eta_i^2< y -  \eta_j^2$ however, we have
	$$ \Rkernel_{12}^{\rm exp, n}(i,  x ; j, y) \xrightarrow[n\to\infty]{} - \int_{\mathcal{C}_{1/4}^{\pi/3}} \frac{\mathrm{d}z}{2\I\pi}
	\exp\big( (z+\eta_j)^3/3 - (z+\eta_i)^3/3  +x(z+\eta_i) -y(z+\eta_j)\big).
	$$
	One can evaluate the integrals above, and we find that in both cases 
	$$ \Rkernel_{12}^{\rm exp, n}(i,  x ; j, y) \xrightarrow[n\to\infty]{} \Rkernel_{12}^{\rm cross}(i,  x ; j, y) = \frac{-\exp\left(\frac{-(\eta_i-\eta_j)^4+ 6(x+y)(\eta_i-\eta_j)^2+3(x-y)^2}{12(\eta_i-\eta_j)}\right)}{\sqrt{4\pi(\eta_j-\eta_i)}}.$$

	As for $\kernel_{22}$, we again decompose the kernel as  $\kernel^{\mathrm{exp}, n}_{22} = \Ikernel^{\mathrm{exp}, n}_{22} + \Rkernel^{\mathrm{exp}, n}_{22}$. For $\Ikernel^{\mathrm{exp}, n}_{22}$, we can write 
		\begin{multline}
		\Ikernel^{\mathrm{exp}, n}_{22}(i,  x ; j, y) = e^{ -\eta_i x -  \eta_j y+ \eta_i^3/3+\eta_j^3/3}  \int_{\mathcal{C}[b_z]}\frac{\mathrm{d}z}{2\I\pi} \int_{\mathcal{C}[b_w]} \frac{\mathrm{d}w}{2\I\pi} \frac{z-w}{z+w}\\ \frac{1}{(-2z+2\sigma^{-1} \varpi n^{-1/3})(-2w+2\sigma^{-1} \varpi n^{-1/3})}
		\exp\Big(n(f(z)+f(w)) \\
		+ n^{2/3}(- \xi\eta_i \log(1-4z^2)- \xi\eta_j \log(1-4w^2))
		\\+ n^{1/3} \xi^2 \eta_i^2 z + n^{1/3} \xi^2 \eta_j^2 w - n^{1/3}\sigma (x z +y w)\Big),
		\label{eq:kernelK22_modifiedcontour_cross}
		\end{multline}
		where  (1) if $\varpi\leqslant0$ then $b_z,b_w>0$, and (2) if $\varpi>0$ then $b_z,b_w\in(0,\sigma^{-1}\varpi)$. 
        Again, by Laplace's method we obtain 
	\begin{multline*}
	\Ikernel^{\mathrm{exp}, n}_{22}(i,  x ; j, y)\xrightarrow[n\to\infty]{}  e^{ -\eta_i x -  \eta_j y+ \eta_i^3/3+\eta_j^3/3} \int_{\mathcal{C}_{b_z}^{\pi/3}}\frac{\mathrm{d}z}{2\I\pi} \int_{\mathcal{C}_{b_w}^{\pi/3}} \frac{\mathrm{d}w}{2\I\pi}
	\frac{z-w}{z+w} \\ \frac{\exp\left(z^3/3 + w^3/3 +4\xi \eta_i z^2/\sigma^2 +4\xi \eta_j w^2/\sigma^2 +\xi^2\eta_i^2 z/\sigma +\xi^2\eta_j^2   w/\sigma- x z -  y w \right)}{(2z-2\varpi)(2w-2\varpi)}  .
	\end{multline*}
 Thus,
	\begin{multline*}
	\Ikernel^{\mathrm{exp}, n}_{22}(i,  x ; j, y) \xrightarrow[n\to\infty]{} \Ikernel^{\rm cross}_{ 22}(i,  x ; j, y) = \\   \int_{\mathcal{C}_{b_z}^{\pi/3}} \frac{\mathrm{d}z}{2\I\pi} \int_{\mathcal{C}_{b_w}^{\pi/3}} \frac{\mathrm{d}w}{2\I\pi}
	\frac{z-\eta_i-w+\eta_j}{4(z-\eta_i+w-\eta_j)}
	\frac{e^{z^3/3 + w^3/3 - x z- y w }}{(z-\varpi-\eta_i)(w-\varpi-\eta_j)}  ,
	\end{multline*}
	where the contours are chosen so that  (1) if $\varpi\leqslant0$ then $b_z>\eta_i$ and $b_w>\eta_j$, and (2) if $\varpi>0$ then $b_z\in(\eta_i,\eta_i+\varpi)$ and $b_w\in(\eta_j,\eta_j+\varpi)$. 

    We next compute the limit of $\Rkernel^{\mathrm{exp}, n}_{22}$ for $x-\eta_i^2>y-\eta_j^2$. When $\varpi>0$, using \eqref{R22 new case1} we have
    \begin{equation*}
	\Rkernel^{\mathrm{exp}, n}_{22}(i,  x ; j, y) \xrightarrow[n\to\infty]{}  
	-\frac{1}{2} \int_{\mathcal{C}_{d_z}^{\pi/3}}\frac{\mathrm{d}z}{2\I\pi}\frac{z\exp\big( (z+\eta_i)^3/3 +(-z+\eta_j)^3/3 -x (z+\eta_i) -y(-z+\eta_j)\big)}{(\varpi+z)(\varpi -z)}.
	\end{equation*}
    When $\varpi<0$, using \eqref{eq:R22expcasalphapetit1} we have
	\begin{multline*}
	\Rkernel^{\mathrm{exp}, n}_{22}(i,  x ; j, y) \xrightarrow[n\to\infty]{}  \\ 
	\frac{-1}{4} \int_{\mathcal{C}_{c_z}^{\pi/3}}\frac{\mathrm{d}z}{2\I\pi} \frac{\exp\big( (z+\eta_i)^3/3 +(\varpi+\eta_j)^3/3 -x (z+\eta_i) -y(\varpi+\eta_j)\big)}{\varpi+z} \\
	+ \frac{1}{4} \int_{\mathcal{C}_{c_z}^{\pi/3}}\frac{\mathrm{d}z}{2\I\pi} \frac{\exp\big( (z+\eta_j)^3/3 +(\varpi+\eta_i)^3/3 -y (z+\eta_j) -x(\varpi+\eta_i)\big)}{\varpi+z} \\
	-\frac{1}{2} \int_{\mathcal{C}_{d_z}^{\pi/3}}\frac{\mathrm{d}z}{2\I\pi}\frac{z\exp\big( (z+\eta_i)^3/3 +(-z+\eta_j)^3/3 -x (z+\eta_i) -y(-z+\eta_j)\big)}{(\varpi+z)(\varpi -z)}\\
	- \frac{1}{4} \exp\big( (-\varpi+\eta_j)^3/3 +(\varpi+\eta_i)^3/3 -y (-\varpi+\eta_j) -x(\varpi+\eta_i)\big)\\
	+ \frac{1}{4}\exp\big( (-\varpi+\eta_i)^3/3 +(\varpi+\eta_j)^3/3 -x (-\varpi+\eta_i) -y(\varpi+\eta_j)\big).
	\end{multline*}
   When $\varpi=0$, using \eqref{R22 new case3} we have
   \begin{multline*}
	\Rkernel^{\mathrm{exp}, n}_{22}(i,  x ; j, y) \xrightarrow[n\to\infty]{}   
	\frac{-1}{4} \int_{\mathcal{C}_{c_z}^{\pi/3}}\frac{\mathrm{d}z}{2\I\pi z} \exp\big( (z+\eta_i)^3/3 + \eta_j^3/3 -x (z+\eta_i) -y \eta_j\big) \\
	+ \frac{1}{4} \int_{\mathcal{C}_{c_z}^{\pi/3}}\frac{\mathrm{d}z}{2\I\pi z} \exp\big( (z+\eta_j)^3/3 + \eta_i^3/3 -y (z+\eta_j) -x \eta_i\big) \\
	+\frac{1}{2} \int_{\mathcal{C}_{d_z}^{\pi/3}}\frac{\mathrm{d}z}{2\I\pi z} \exp\big( (z+\eta_i)^3/3 +(-z+\eta_j)^3/3 -x (z+\eta_i) -y(-z+\eta_j)\big) \\
	- \frac{1}{4} \exp\big( \eta_i^3/3 +\eta_j^3/3 -\eta_ix-\eta_jy\big).
	\end{multline*} 	
    The above contours are chosen so that $c_z>-\varpi$ and (1) if $\varpi\neq0$ then $d_z$ is between $-\varpi$ and $\varpi$, and (2) if $\varpi=0$ then $d_z>0$. 
    By deforming the contours $\mathcal{C}_{c_z}^{\pi/3}$ past $-\varpi$, the above equations match with $\Rkernel_{22}^{\rm cross}$ given by \eqref{cross R22}, in which  $c_z<-\varpi$.  When  $ x -   \eta_i^2< y -  \eta_j^2$, $ \Rkernel_{22}^{\rm exp, n}$ is determined by antisymmetry.

At this point, we have shown that when $\alpha=1/2$ and for any set of points $ \lbrace  i_r, x_{i_r} ; j_s, x_{j_s}  \rbrace_{1\leqslant r,s\leqslant k} \in \lbrace 1, \dots, k\rbrace \times \R$, 
\begin{equation*}
 \Pf\Big(\kernel^{\rm exp, n}\big(i_r, x_{i_r} ; j_s, x_{j_s} \big) \Big)_{r,s=1}^k \xrightarrow[q\to 1]{} \Pf\Big({\kernel}^{\rm cross}\big(i_r, x_{i_r} ; j_s, x_{j_s}\big) \Big)_{r,s=1}^k.
\end{equation*}
In order to conclude that the Fredholm Pfaffian likewise has the desired limit, one needs a control on the entries of the kernel $\kernel^{\rm{exp}, n}$, in order to apply dominated convergence.
\begin{lemma}
Let $a\in \R$ and $0\geqslant\eta_1< \dots <\eta_k$ be fixed. There exist positive constants $C, c, m$ for $n>m$ and $x, y>a$,
	\begin{align*}
	\Big\vert  \kernel^{\rm exp, n}_{11}(i, x;j,y)\Big\vert &< C \exp\big(-c x - c y \big),\\
	\Big\vert  \kernel^{\rm exp, n}_{12}(i, x;j,y)\Big\vert &< C \exp\big(-c x \big),\\
	\Big\vert  \kernel^{\rm exp, n}_{22}(i, x;j,y)\Big\vert &< C.
	\end{align*}
	\label{lem:expobound}
\end{lemma}
\begin{proof}
	The proof is very similar to that of Lemmas 5.11 and 6.4 in \cite{baik2017pfaffian}. Indeed, using the same approach as in the proof of these lemmas, we obtain that 
		\begin{align*}
		\Big\vert  \Ikernel^{\rm exp, n}_{11}(i, x;j,y)\Big\vert &< C \exp\big(-c x - c y \big),\\
		\Big\vert  \Ikernel^{\rm exp, n}_{12}(i, x;j,y)\Big\vert &< C \exp\big(-c x \big),\\
		\Big\vert  \Ikernel^{\rm exp, n}_{22}(i, x;j,y)\Big\vert &< C\exp\big(-c x - c y \big),
		\end{align*}
	and 
		\begin{align*}
		\Big\vert  \Rkernel^{\rm exp, n}_{11}(i, x;j,y)\Big\vert &=0,\\
		\Big\vert  \Rkernel^{\rm exp, n}_{12}(i, x;j,y)\Big\vert &\leqslant C \mathds{1}_{i<j}  \exp\big( (x+y)(\eta_i-\eta_j)\big),\\
		\Big\vert  \Rkernel^{\rm exp, n}_{22}(i, x;j,y)\Big\vert &< C.
		\end{align*}
	Recall that when $i<j$, $\eta_i-\eta_j<0$, so that the bounds on $\Ikernel^{\rm \exp, n}$ and $\Rkernel^{\rm \exp, n}$ combine together to the statement of Lemma \ref{lem:expobound}. 
\end{proof}
The bounds from Lemma \ref{lem:expobound} are such that the hypotheses in Lemma \ref{lem:hadamard} are satisfied. We conclude, applying dominated convergence in the Pfaffian series expansion,  that
	$$ \lim_{n\to\infty} \PP\left(\bigcap_{i=1}^k \left\lbrace H_n(\eta_i) < x_i  \right\rbrace\right)  = \Pf\big( J- \kernel^{\rm  cross}\big)_{\mathbb{L}^2(\mathbb{D}_k(x_1, \dots, x_k))}.$$

\subsection{Proof of Theorem \ref{theo:SU}}

	The proof is very similar as that  of Theorem \ref{theo:crossfluctuations}. We use a similar  rescaling of the kernel: we define the rescaled kernel 
		\begin{multline*}
		\kernel^{\mathrm{exp}, n}(i,  x_i ; j, x_j):= \\
		{\footnotesize \begin{pmatrix}
		\varpi^{-2}\sigma^2 n^{2/3}e^{ \eta_i x_i +  \eta_j x_j - \eta_i^3/3-\eta_j^3/3}\kernel_{11}^{\rm exp}\big(i,X_i; j, X_j\big) & \ \ \varpi^{-1}\sigma n^{1/3}e^{ \eta_i x_i -  \eta_j x_j - \eta_i^3/3+\eta_j^3/3} \kernel_{12}^{\rm exp}\big(i,X_i; j, X_j\big)\\
		\varpi^{-1}\sigma n^{1/3}e^{ -\eta_i x_i +  \eta_j x_j+ \eta_i^3/3-\eta_j^3/3} \kernel_{21}^{\rm exp}\big(i,X_i; j, X_j\big) &\ \ \varpi^{2} e^{ -\eta_i x_i -  \eta_j x_j+ \eta_i^3/3+\eta_j^3/3}\kernel_{22}^{\rm exp}\big(i,X_i; j, X_j\big)
		\end{pmatrix}},\end{multline*}
	 Then, we decompose the kernel as  $\kernel^{\mathrm{exp}, n}(i,  x_i ; j, x_j)= \Ikernel^{\mathrm{exp}, n}(i,  x_i ; j, x_j)+ \break \Rkernel^{\mathrm{exp}, n}(i,  x_i ; j, x_j)$   using the formulas (and choice of contours) of Section \ref{sec:kernelexp} in the case $\alpha>1/2$. Thus, the formulas are slightly simpler than in the proof of  Theorem \ref{theo:crossfluctuations}. To show that the kernel $\kernel^{\mathrm{exp}, n}$ converges pointwise to $\kernel^{\mathrm{SU}}$ ,  we follow the same steps as in the proof of Theorem \ref{theo:crossfluctuations} as if $\varpi=+\infty$ and contours $\mathcal{C}[a_z], \mathcal{C}[a_w], \mathcal{C}[b_z], \mathcal{C}[b_w]$ are chosen to be consistent with the constraints on contours in the case $\alpha>1/2$ of Section \ref{sec:kernelexp}. Finally, the kernel satisfies the same uniform bounds as in Lemma \ref{lem:expobound}, so that we conclude the proof by dominated convergence as above. 

\bibliographystyle{amsalpha}
\bibliography{ftasep.bib}

\end{document}